\documentclass[12pt]{article}
\usepackage{verbatim}
\usepackage{amssymb,amsmath, amsthm,bm}
\usepackage{graphicx}
\usepackage{epsfig}
\newtheorem{theorem}{Theorem}
\newtheorem{corollary}{Corollary}[section]
\newtheorem{lemma}[corollary]{Lemma}
\newtheorem{proposition}[corollary]{Proposition}

\newcommand{\Prob} {{\bf P}}
\newcommand{\Z}{{\mathbb Z}}
\newcommand{\E}{{\bf E}}
\newcommand{\Es}{{\rm Es}}
\newcommand{\Cp}{{\rm cap}}
\newcommand{\R}{{\mathbb{R}}}

\newcommand{\dist}{{\rm dist}}
\newcommand{\x}{{\bf x}}

\def \p {\partial}

\def \diam {{\rm diam}}

\def \eset {\emptyset}

\def \saws {{\cal W}}

\def \loop {{\cal L}}

\def \F {{\mathcal F}}

\def \cone {{\mathcal C}}

\newcommand  \Sep  {\textsf{Sep}}
\def \pairs {{\mathcal A}}
\def \barpairs {{\bar \pairs}}

\def \e {{\bf e}}

\def \barsaws {{\bar \saws}}

\def \bgamma {{\bm \gamma}}
\def \bfeta  {{\bm \eta}}

\def  \hm  {{\rm hm}}
\def \seppair {{\mathcal K}}

\newenvironment{definition}[1][Definition]{\begin{trivlist}
\item[\hskip \labelsep {\bfseries #1}]}{\end{trivlist}}

\def \linehere { {\hrule}}
\def \labove { \mtwo \linehere \linehere \linehere \ms   }

\def \lbelow {{\ms \linehere \linehere \linehere \mtwo}}

\def \mtwo {{\medskip \medskip}}
\def \ms {{\medskip}}
\newcommand {{\whoknows}} {{\mathcal C}}

\newenvironment{advanced}
{ \labove \begin{quote} \begin{small}}
{ \end{small}\end{quote} \lbelow }

\def \begad{\begin{advanced}}
\def \endad{\end{advanced}}

\title{The infinite
two-sided loop-erased random walk}
\author{Gregory F. Lawler\\Department of Mathematics\\University of Chicago\thanks{Research supported by NSF grant 
DMS-1513036.}}

\begin{document}
\maketitle
\begin{abstract}
The loop-erased random walk (LERW) in $ \Z^d,
d \geq 2$,
 is obtained
by erasing loops chronologically from simple random walk.  In this
paper we show the existence of the two-sided LERW which can be
considered as the distribution of the  LERW as seen by a point
in the ``middle'' of the path.
\end{abstract}

\section{Introduction}

In this paper we establish the existence
of the infinite two-sided loop-erased random walk (LERW).
We start by explaining what this means.
The (infinite, one-sided)
LERW  is the measure on
non self-intersecting paths obtained by erasing loops
chronologically from a simple random walk.  This sentence
can be taken literally for $d \geq 3$, but for $d =2$
we need a little care.  We will give a 
definition that is valid for $d \geq 2$ that is
easily seen to be equivalent to the usual definition
for $d \geq 3$.
A simple random walk conditioned to never return
to the origin is a simple random walk weighted by
the Green's function (for $d \geq 3$) or the potential
kernel (for $d=2$). 
\begin{itemize}
\item  \textbf{Loop erasure}. Let $S_j$ denote
a simple random walk starting at the origin in
$\Z^d, d\geq 2$, conditioned to never return 
to
the origin.     Let
$\sigma_0 = 0$, and for $n > 0$, let
$\sigma_n = \max\{j: S_j = S_{\sigma_{n-1} +1}\}. $
Then the {\em (infinite, one-sided) loop-erased random walk
(LERW)}  $\hat S_n$ is defined by
\[   \hat S_n = S_{\sigma_n} = S_{\sigma_{n-1} + 1}. \] 
\end{itemize}
For $d \geq 3$, one gets the same measure by taking
a simple random walk without conditioning and defining
$\sigma_0 = \max\{j: S_j = 0\}$.
This is the original definition as in \cite{LLERW}, but it is often
useful to view
this probability measure on infinite self-avoiding paths as a consistent
collection of measures on finite paths.   
We say that $\eta = [\eta_0,\eta_1,\ldots,\eta_n]$ is a
self-avoiding walk (SAW) if it is a nearest neighbor path
with no self-intersections.  We will reserve the term SAW
for finite paths.  The following two facts can be readily
derived from the definition by considering the unique
decomposition of a  simple random walk  $\omega$
path starting at the origin and never returning
to the origin
whose loop-erasure  is $[\eta_0,\eta_1,\ldots]$ as
\begin{equation}  \label{jun24.1}
\omega =  [\eta_0,\eta_1]\oplus
l_1 \oplus [\eta_1,\eta_2] \oplus l_2 \cdots, 
\end{equation}
where $l_j$ is a loop rooted at $\eta_j$ that does
not visit $\{\eta_0,\ldots,\eta_{j-1}\}$, and $\oplus$
denotes concatenation.

\begin{itemize}

\item \textbf{ Laplacian random walk}.  
Suppose
$\eta = [\eta_0,\ldots,\eta_n]$ is a  SAW 
in $\Z^d$ starting at the origin.  Then,

\[  \Prob\left\{\hat S_{n+1} = z \mid [\hat S_0,\ldots,\hat S_n]
= \eta \right\} = \frac{g_\eta(z)}{2d\,\Es_\eta(\eta_n) },  \] where
\[   \Es_\eta(\eta_n) = \Delta g_\eta(z) = \frac{1}{2d}\sum_{|z-\eta_n| = 1} g_\eta(z)\]
and
\begin{itemize}
\item ($d \geq 3$)  $g_\eta(z)$ is the probability that a random walk starting
at $z$ never visits $\eta$, that is, the unique function that
is (discrete) harmonic  
on $\Z^d \setminus \eta$; vanishes on $\eta$; and has boundary value
$1$ at infinity.
\item ($d=2$)   
$g_\eta(z)$ is the unique function that
is (discrete) harmonic  
on $\Z^2 \setminus \eta$; vanishes on $\eta$;  and  satisfies
\[    g_\eta(z) \sim \frac 2{\pi} \log |z|, \;\;\;\; z \rightarrow \infty.\]
\end{itemize}

\item \textbf{Loop measure formulation}.  If $\eta$ is a SAW
starting at the origin, then
\begin{equation}  \label{consistent}
\Prob\{[\hat S_0,\ldots,\hat S_n]
= \eta \} = (2d)^{-n} \, G_0\,F_\eta \, \Es_\eta(\eta_n) , 
\end{equation}
where $G_0 = 1$ if $d = 2$ and $G_0 = G_{\Z^d}(0,0) < \infty$
if $d\geq 3$, and 
\[   F_\eta = \prod_{j=1}^n G_{A_j}(\eta_j,\eta_j).\]
Here $A_j = \Z^d \setminus \{\eta_0,\ldots,\eta_{j-1}\},$
and $G_{A_j}(\cdot,\cdot)$ denotes the simple random walk Green's
function in $A_j$.  That is, $G_{A_j}(x,y)$ is the expected number
of visits to $y$ of a random walk starting at $x$ killed upon
leaving $A_j$.
An alternative expression for $F_\eta $ 
(see Section \ref{loopsec} for definitions) is
\[    F_\eta = F_\eta(\Z^d \setminus \{0\})
=\exp \left\{\sum_{\ell \subset \Z^d \setminus \{0\},
\ell \cap \eta \ne \eset }
m(\ell) \right\}, \]
where $m$ denotes the random walk loop measure. 
If $d \geq 3$, we can write
\[   G_0 \, F_\eta =
F_\eta(\Z^d  )
=\exp \left\{\sum_{\ell \subset \Z^d, \ell \cap \eta \ne \eset }
m(\ell) \right\}, \]
but the right-hand side is infinite for $d=2$.
\end{itemize}

\begad

What will make this paper a little complicated to write is that we will do both the $d=2$ and $d=3$ cases simultaneously.  The basic idea of the coupling argument is the same in both cases but the details about the LERW and the loop measures differ.  In three dimensions the random walk is transient and this will give  us some useful
estimates.  The random walk in two dimensions is recurrent so   these estimates will
not be available.  However, planarity gives us another set of tools.  Random walks can make loops about the origin and disconnect the origin from infinity.  Also, there is the "Beurling estimate" that tells us that if we are close to any continuous path then there is a good chance that a random walk will hit it before going too far.

\endad

The  consistent family
of measures on (finite) SAWs 
given by \eqref{consistent} 
represents the probability that the LERW starts with $\eta$.  This  can be
called the {\em one-sided} measure because the LERW continues
on only one   side of $\eta$.  In this paper we will consider
the two-sided measure which can be viewed as the 
distribution of the ``middle'' of
a LERW.  There are two versions of
our result. One is in terms of LERW excursions
as discussed in \cite{LL}.   Suppose $A$ is 
a simply connected subset of $\Z^d$ containing
the origin and $x,y$ are distinct points in $\p A$. 
Consider simple random walks starting at $x$ conditioned
to enter $A$ and then leave $A$ for the first time at
$y$ at which time they are stopped.  Erase loops, restrict to the event
that the loop-erasure goes through the origin,
and then normalize to make this a probability measure
which we denote by 
$\lambda_{A,x,y}^\#$.  It is supported  on $\pairs_n(A;x,y)$,
the set of SAWs from $x$ to $y$ in $A$ going through
the origin. 
If $\eta = [z_0,z_1,\ldots,z_m], \gamma =[\gamma_0,
\gamma_1,\ldots,\gamma_k]$,   are SAWs 
we write $\eta \prec \gamma$ if $\gamma$ includes $\eta$ 
in the sense that  there exists $j$ such that $z_{j+i}
= \gamma_i, i=0,1,\ldots,m $.  We write $\pairs_{A,x,y}( \eta) $
for the corresponding sets of SAWs that include $\eta$.
While our main theorem discusses this measure, we will first study a slightly different measure.

Let
\[    C_n = \{x \in \Z^d: |x| < e^n\} , \]
with boundary $\p C_n = \{x \in \Z^d: \dist(x,C_n) = 1\}$.
  Let  $\pairs_n$  be the union of $\pairs(C_n;x,y)$
over all $x,y \in \p C_n$.   We write
$\pairs _n (\eta)$
for the corresponding set of SAWs that include $\eta$.  Note
that $\pairs_n$ is $\pairs_n(\eta)$ for the trivial SAW
$\eta = [0].$

There is another way to describe the set $\pairs_n$.
Let  $\saws_n$ denote the set of SAWs starting at the origin, ending
at $\p C_n$, and otherwise staying in $C_n$.  
Then we can
also define  $\pairs_n$ to be  the set of ordered pairs
$\bfeta = (\eta^1,\eta^2) \in \saws_n^2 = \saws_n
\times \saws_n$ with  
$\eta^1 \cap \eta^2 = \{0\}$.  This is essentially the 
same definition of $\pairs_n$ as above if we use the natural bijection  given
by $\eta \leftrightarrow (\eta^1)^R \oplus \eta^2$, where
$R$ denotes path reversal.
We define $\mu_n$ to be the probability
measure on $\saws_n$ induced by
the infinite LERW by stopping the path at the first visit to
$\p C_n$; we will also write $\mu_n$ for $\mu_n \times
\mu_n$,  the product measure
on $\saws_n^2$.  We define a measure $\lambda_n$
on $\pairs_n$ by stating that its Radon-Nikodym derivative
with respect to $\mu_n$ is
\[    1\{\bfeta \in \pairs_n\} \, \exp\{-L_n(\bfeta)\}, \]
where $L_n(\bfeta)$ denotes the loop measure
(see Section \ref{loopsec})  of loops
in $C_n \setminus \{0\}$ that intersect both $\eta^1$
and $\eta^2$.   For $d=2$, we will restrict
to loops that do not disconnect the origin from infinity.

If $k < n$  and $\bfeta \in \pairs_k$, we let
$\pairs_n(\bfeta)$ be the set of $\bgamma \in \pairs_n$
that are extensions of $\bfeta$  (we write $\bfeta \prec
\bgamma$).   Similarly, if $A \supset C_{n+1}$,
we let $\pairs_{A,x,y}[\bfeta]$ be the set of $\bgamma
\in \pairs_n(A,x,y)$ that are extensions of
$\bfeta$.  We can state our main theorem.

\begin{theorem} \label{introtheorem}
There exists $\alpha > 0$ such that 
the following holds.  For every positive integer
$k$ and every   $\bfeta \in \pairs_k$,
the limit
\[  p(\bfeta) = \lim_{n \rightarrow \infty}
\frac{  \lambda_n[\pairs _n( \bfeta)]}{ \lambda_n[\pairs _n] }
\]
exists.  In fact, 
\begin{equation}  \label{jun12.3}        \lambda_n[\pairs _n( \bfeta)]
= p(\bfeta) \, \lambda _n[\pairs _n] \, \left[1 + O(e^{\alpha(k-n)})\right].
\end{equation}
Moreover, if $A$ is a simply connected set
containing  $C_{n+1}$ and $x,y \in \p A$
with $\pairs(A;x,y)$ nonempty,
then
\begin{equation}  \label{jun12.3.alt} 
\lambda_{A,x,y}^\#[\pairs_{A,x,y}(\bfeta)] = 
   p(\bfeta) \, \left[1 + O(e^{\alpha(k-n)})\right].
   \end{equation}
\end{theorem}

The statement \eqref{jun12.3} uses a convention that we will
use throughout this paper. 
Any implicit constants arising from
$O(\cdot)$ or $\asymp$ notations can depend on $d$ but
otherwise are assumed to be uniformly bounded
over all the parameters including vertices in $\Z^d$, $k,n$, and  
$\bfeta \in \pairs_k$.   In other words, we can say that there exists $c,\alpha$ such
that for all $k \leq n-1$ and all $\bfeta \in \pairs_k$, 
\[         \left|\log\left(\frac{ \lambda_n [\pairs _n( \bfeta)]}
{p(\eta) \, \lambda_n[\pairs _n]} \right)\right| \leq c \, e^{\alpha(k-n)}.\]

We will only do the details of the  proof for $d=2$ and $d=3$ for which 
the result is new. For $d=4$, it can be derived
from the construction of the two-sided walk in \cite{LSunW}, and for $d\geq 5$, it is
even easier.   The proof we give for  $d =3$ can be adapted easily
to $d > 3$; it uses   transience of the random walk.  
The $d=2$
case is similar but there are difficulties arising
from recurrence of the random walk.     These can be overcome by making use
of planarity, and, in particular, the Beurling estimate and the
disconnection exponent for two-dimensional random walks.

If we let $p_k$ be $p$ restricted to $\pairs_k$, then $\{p_k\}$
is a consistent family of probability measures and induces a probability
measure on pairs  of
infinite self-avoiding paths starting
at the origin that do not intersect (other than the initial point).  We
call this process the {\em two-sided infinite loop-erased random walk}. 
The result in this paper is different (and, frankly, easier) than questions about the scaling limit  of  loop-erased walk.   For the one-sided case in  $d=2$,
the scaling limit is now  well understood as the Schramm-Loewner evolution
with parameter $\kappa = 2$; see \cite{BLV,Kenyon,LSW,LV} for some of the main
papers here. The existence of a scaling limit in $d \geq 3$
was established in some sense by Kozma \cite{Kozma} but there is still much work to do in understanding the limit and the rate of convergence to the limit; see \cite{Shir} for some work in this
direction.  Our methods are similar to those
in \cite{Nonintersect,Vermesi} where convergence to the measure on mutually non-intersecting Brownian motions is studied and in \cite{Masson} where two-dimensional loop-erased walk is studied.

We now outline the paper.  Most of the work is
focused on \eqref{jun12.3}; the final section will show how to derive \eqref{jun12.3.alt} from this.
We start by giving the notations that we will  use and state the main theorem we will prove. 
  There are two different ways to define ``loop-erased walk stopped when it leaves the ball of radius $r$'': one can either take a simple random walk and stop it when it leaves the ball and then erase loops
  or    erase loops from the infinite random walk and then stop the loop-erased walk when it leaves the disk.  It turns out easier to analyze the latter case first and this is where we focus our effort.  We use the ``loop measure'' description of the LERW and we review relevant facts about the loop measure in Section  \ref{loopsec}.  Here there is a difference
between two and higher dimensions.  In two dimensions, the measure of loops intersecting a finite set is infinite.  However, this can be handled by splitting the set of loops into those that surround the origin and those that do not.  
Roughly speaking, the loops that surround the origin give the divergence in the loop measure; however, all such loops intersect all infinite self-avoiding paths
starting at the origin, so this term cancels out. 

To use the approach in \cite{Nonintersect,Vermesi}, one needs two ``obvious'' lemmas about random walks.  Although proofs of these results appear several other places, we choose to give sketches here as well.
The first is in Section \ref{rwlemmasec} where it is shown that a random walk starting on the sphere of radius $R$ about the origin, stopped when it reaches distance $r$ from its starting point, conditioned to avoid some set contained in the ball of radius $R$, has a reasonable chance of ending up at
a point distance $R + (r/2)$ away from the origin.  If there is no
conditioning, this follows from the central limit theorem, and the conditioning
should just increase the probability (this is why it is ``obvious'').   It is very useful to know that one can find uniform bounds, uniform over $R,r$, the starting point of the walk, and the avoidance set.  The second ``obvious'' fact is called a {\em separation lemma} which roughly states that 
loop-erased walks conditioned to avoid each other tend to stay far
apart.  Again, the key is to find a version of this that is uniform.  
The proof uses the same basic idea as the original proof in
\cite{cutpoint} as adapted for loop-erased walk \cite{Masson,LV,Shir}.
Section \ref{lerwsec} goes over facts about LERW including an important
lemma that the LERW stopped about reach radius $R$ is ``independent
up to constants'' with the walk starting at the last visit to the disk
of radius $2R$.  The next subsection considers pairs of walks and sets up for Section \ref{couplesec} where the coupling is done.  As has been done in several of the papers before, one chooses a large integer $N$ and
considers the probability measure on pairs of SAWs given by the
LERWs weighted by a loop measure term.  We then view this measure
as giving transition probabilities for LERW stopped when it reaches
a smaller radius $n$ and this is the process that we couple.

\section{The main theorem}

\subsection{Notation and main result}

We list the notation that we will use.

\begin{itemize}

\item  If $A \subset \Z^d$, we write
$\hat A = A \setminus \{0\}$.  In particular,
$\hat \Z^d = \hat \Z^d \setminus \{0\}.$  We write
\[  \p A = \{z \in \Z^2: \dist(z,A) = 1\}, \;\;\;\;
 \p_i A = \p(\Z^2 \setminus A)
 = \{z \in A: \dist(z,\Z^2 \setminus A) = 1\}. \]

\item  If $z \in \partial A$, we
 let $H_A(x,z)$ denote the Poisson kernel, that is,
 the probability that a simple random walk starting at $x$
 first visits $\Z^2 \setminus A$ at $z$.  If $x \in \p
 A$, then $H_A(x,z) = \delta(x-z)$.  If $w,z$ are 
 distinct points in $\p A$, we let $H_{\p A}
  (w,z)$ denote the boundary Poisson kernel
 defined by
 \[    H_{\p A}(w,z) = \frac 1{2d}
 \sum_{x \in A, |w-x| = 1}
    H_A(x,z).\]
    A last-exit decomposition  shows that if $x \in A$, then
    \[  H_A(x,z) = G_A(x,x) \, H_{\p(A \setminus \{x\})}
       (x,z).\]

\item If $n \geq 0$ (not necessarily an integer), let
\[   C_n = \{z \in \Z^d: |z| < e^n \} , \]
be the discrete ball of radius $e^n$.
Note that $C_0 = \{0\}$ and $\hat C_n
= C_n\setminus C_0$.    

\item Let $\e_n = (e^n,0,\ldots,0)$
be the element of $\R^d$ with first component $e^n$ and all other
components equal to zero. 

\item  If $S$ is a simple
random walk,  then we write $\hat S$
for its (chronological) loop-erasure.   

\item If $\eta = [\eta_0,\ldots,\eta_j]$ is a SAW, we write $|\eta| = j$
for the number of steps in $\eta$. We call $\eta_0$ and
$\eta_j$ the {\em initial} and {\em terminal} points or
vertices of $\eta$,
respectively.

\item  $\saws_n$ is the set of SAWs 
starting at the origin whose terminal point is in $\p C_n$ and
all other vertices are in $C_n$.

\item  $\barsaws_n$ is the set of infinite self-avoiding paths
whose initial point is in $C_n$ and all other vertices are in
$\Z^d \setminus C_n$.  In particular, $\barsaws_0$ is the
set of infinite self-avoiding paths starting at the origin.

\item  If $n  \leq  m$, then $\barsaws_{n,m}$ is the set of
SAWs satisfying: the initial point is in  $C_n$,  
the terminal point is in $\p C_m$, and all other vertices
are in $C_m \setminus C_n$.

\item We write $\eta \prec \tilde \eta $ if $\eta$
is contained in $\tilde \eta$.  If $\eta, \tilde \eta$
both start at $0$, then this means that  $\eta$ is an initial
segment of $\tilde \eta $.

\end{itemize}

If $\eta \in \barsaws_0$ and $n > 0$, there is a
unique decomposition
\begin{equation}
\label{decomp}  \eta = \eta_n \oplus \eta^* \oplus \bar \eta_{n+1} , 
\end{equation}
where $\eta_n \in \saws_n, \bar \eta_{n+1} \in \barsaws_{n+1}$.
Similarly, 
if $n \leq m-1$ and $\eta \in \saws_m$,  there
is a unique decomposition
\begin{equation}
\label{decomp2}  \eta = \eta_n \oplus \eta^* \oplus \eta_{n+1,m} , 
\end{equation}
where $\eta_n \in \saws_n, \eta_{n+1,m} \in \barsaws_{n+1,m}$.
In these
decompositions   $\eta^*$
is a SAW starting at $\p C_n$ with  
terminal vertex  in $\p_i C_{n+1}$.  In \eqref{decomp2}
we also need $\eta^* \subset C_m$.

We will also be considering pairs of paths.  We will use
bold-face notation for pairs of SAWs.

\begin{itemize}
\item  $\pairs_n$ is the set of  ordered
pairs $\bfeta = (\eta^1,\eta^2) \in \saws_n^2:= \saws_n
\times \saws_n$ such that
\[  \eta^1 \cap \eta^2 = \{0\}.\]

\end{itemize}

\begad
The notation gets a little cumbersome, but it useful to remember that bold-face
$\bgamma,\bfeta$ will always refer to ordered pairs of SAWs.
\endad

Recall in the introduction that
we  wrote $\pairs_n$ for the set
of SAWs $\eta$ whose initial and terminal vertices
are in $\p C_n$; all other vertices are in $C_n$;
and that include the origin as a vertex. 
Indeed, there is a simple bijection to show
that these are essentially the same set: 
\[   (\eta^1,\eta^2) \longleftrightarrow \eta =
(\eta^1)^R \oplus \eta^2 , \]
where $R$ denotes the reversal of the walk.

\begin{itemize}

\item  If $1 \leq n \leq m-1$,
$\barpairs_{n,m}$ is the set of  ordered
pairs $\bfeta = (\eta^1,\eta^2) \in \barsaws_{n,m}
\times \barsaws_{n,m}$
with $\eta^1 \cap \eta^2 = \eset$.

\item  We write $(\eta^1,\eta^2) \prec
(\tilde \eta^1,\tilde \eta^2)$ if
$\eta^1 \prec \tilde \eta^1$ and
$\eta^2 \prec \tilde \eta^2$.

\end{itemize}

\begin{itemize}

\item  Let
$\mu_m$ denote the probability measure on 
$\saws_m$ obtained
by taking an  infinite loop-erased  random walk  and truncating
the path at the first visit to $\p C_m$.  As a slight abuse
of notation, we also write 
$\mu_m$ for the product measure $\mu_m \times \mu_m$
on  $ \saws_m^2$.

\item  If $\bgamma  =(\gamma^1,\gamma^2)
\in\saws_m \times \saws_m$, we define
\[    Q_m(\bgamma) = \exp \left\{-L_m (\gamma^1,\gamma^2)
\right\} = 
1\{\bgamma \in \pairs_m \} \,  \exp \left\{-L_m (\gamma^1,\gamma^2)
\right\},
\]
where:
\begin{itemize}
\item $ L_m (\gamma^1,\gamma^2)  = \infty$
if $\bgamma \not \in \pairs_m$,
\end{itemize}
and if $\bgamma \in \pairs_m$,
\begin{itemize}
\item $(d \geq 3$)   $L_m (\gamma^1,\gamma^2)$
denotes the loop measure (see Section \ref{loopsec})
of the set  of loops
in $\hat C_m$ that intersect both $\gamma^1$ and $\gamma^2$. 
\item $(d = 2)$
$L_m (\gamma^1,\gamma^2)$
denotes the loop measure 
of the set  of loops
in $\hat C_m$ that intersect both $\gamma^1$ and $\gamma^2$   and do not disconnect $0$
from $\p C_m$.
\end{itemize}
\end{itemize}

\begad
For $d=2$, we will be ignoring loops in $\hat C_n$ that disconnect $0$
from $\p C_n$.  The reason is that all such loops intersect  all 
$\gamma \in \saws_n$ and hence these loops have no effect on the probability
distribution obtained by tilting by $e^{-(\rm loop \;term)}$.  Restricting to
loops in two dimensions that do not intersect will give us estimates analogous
to estimates in three dimensions obtained from transience of
the random walk.

\endad

\begin{itemize}

\item  Let
\[ \lambda_m   = \E_{\mu_m} \left[Q_m(\bgamma)\right]=
\sum _{\bgamma \in \saws_m^2} \mu_m(\bgamma)
\, Q_m(\bgamma), \]
and if $n \leq m-1$ and $\bfeta  \in \saws_{n}^2$, we let
\[    \lambda_m(\bfeta) = \sum _{\bgamma \in \saws_m^2,\;
\bfeta \prec \bgamma}
   \mu_m(\bgamma)
\, Q_m(\bgamma)  .\]
Note that   
$\lambda_m(\bfeta)$ is nonzero only if $\bfeta \in \pairs_n$ and
that 
\[   \lambda_m = \sum_{\bfeta \in \pairs_n} \lambda_m(\bfeta).\]
\end{itemize}

To prove  \eqref{jun12.3}, it suffices to prove
that for all $\bfeta \in \pairs_n$ and $m \geq n+1$,
\begin{equation}  \label{jun12.4}
\frac{\lambda _{m+1}(\bfeta)}
   {\lambda_{m}(\bfeta) } = \frac{\lambda_{m+1}}
   { \lambda_{m}} \, \left[1 +O(e^{\alpha(n-m)})\right], 
   \end{equation}
since this implies that
\[    \frac{\lambda _{m+1}(\bfeta)}{\lambda_{m+1}}
   = \frac {\lambda_{m}(\bfeta) }
   { \lambda_{m}} \, \left[1 +O(e^{\alpha(n-m)})\right]. \]
We concentrate on \eqref{jun12.4} and use a coupling argument
to establish this.  Let $ \lambda_m^\#,\lambda_m^\#(\bfeta)$ denote
the probability measures on $\pairs_m$ whose Radon-Nikodym
derivative with respect to $\mu_n$ are 
\[    \frac{Q_m(\bfeta')}{\lambda_m}, \;\;\;\;
\frac{Q_m(\bfeta') \, 1\{\bfeta \prec \bfeta'\}}
  {\lambda_m(\bfeta)}, \]  respectively.  We show
that we can couple $\pairs_m$-valued
random variables with
distributions
$ \lambda_m^\#$ and $\lambda_m^\#(\bfeta)$ on the same probability
space so that, except for an event of small probability, the
paths agree except for an initial part of the path. (It would
be impossible to couple them so that the total paths agree since 
$\lambda_m^\#(\bfeta)$ is supported on pairs of walks that start
with $\bfeta$.)

\subsection{Some results about two-dimensional walks}

We will assume that the reader is acquainted with basic facts about simple random walk; we will use  \cite{LL} as a reference.
For $d \geq 3$, we will be using transience of the
random walk; in particular, we will use the estimate that
if $z \in \Z^d$, then the probability that a   random walk
starting at $z$
gets within distance $r$ of the origin is bounded above by
$c (r/|z|)^{d-2}$.  For $d=2$, some of the important results are perhaps less known, so we will review them here.  This subsection can be skipped at first reading
and referred to as necessary. 

In this subsection we let
$S_j$ denote a simple random walk,
and
\[  \rho_n = \min\{j: S_j \not \in C_n\}.\]

We let $a(x)$ be the potential kernel in $\Z^2$; it can be described
as the unique function that is harmonic on $\Z^2 \setminus \{0\}$; vanishes
at the origin; and is asymptotic to $(2/\pi) \log |x|$.  It is known 
\cite[Theorem 4.4]{LL} that
$a(x) = 1$ for $|x| = 1$ and 
\begin{equation}  \label{potential}
 a(x) = \frac 2{\pi} \, \log |x| + k_0 + O(|x|^{-2}), \;\;\;\;
|x| \rightarrow \infty ,
\end{equation}
for a known  constant $k_0$ (whose value is not important to us).
In particular, 
\[     a(x)   = \frac{2n}{\pi} + k_0 + O(e^{-n}),\;
   \;\;\; x \in \p C_n.\]
More generally, if $\eta$ is a SAW (or any finite set), we define
$a_\eta(x)$ to be unique function that is harmonic on 
on $\Z^2 \setminus \eta$; vanishes on $\eta$;
and is asymptotic to $(2/\pi) \log |x|$ as $|x| \rightarrow \infty$.   It is related to escape
probabilities by
\[            a_\eta(x) = \lim_{n \rightarrow \infty}
 \frac{2 n}{\pi} \,  \Prob^x\{S[0,\rho_n] \cap \eta
 =  \eset\}  ,\]
 see \cite[Proposition 6.4.7]{LL}.
 If $\eta$ contains $0$, we can write
 \begin{equation}  \label{sep13.4}
 a_\eta(x) = a(x) - \E^x[a(S_\tau)], 
\end{equation}
 where
$  \tau = \tau_\eta = \min\{j \geq 0: S_j \in \eta\}.$
If $x \in \eta$, we write
\[  \Es_\eta(x) =  \lim_{n \rightarrow \infty}
 \frac{2 n}{\pi} \,  \Prob^x\{S[1,\rho_n] \cap \eta
 =  \eset\} = \Delta a_\eta(x), \]
 where $\Delta$ denotes the discrete Laplacian.
 The capacity of $\eta$, $\Cp(\eta)$, is
 defined by
 \[  a_\eta(x) = \frac 2 \pi \, \log |x| -
  \Cp(\eta) +o(1), \;\;\;\; |x| \rightarrow \infty.\]

Random walk in $\Z^2$ conditioned to avoid $\eta$ is the $h$-process
obtained from the function $a_\eta(x)$.  In other words, if $x$
is in the unbounded component of $\Z^2 \setminus \eta$, then
the transition probabilities are given by
\[         p(x,y) =   \frac{a_\eta(y)}{4 \, a_\eta(x)}.\]
This process can also be started on the boundary of the
unbounded component (that is, on points of $\eta$ that are
connected to infinity in $\Z^2 \setminus \eta$), by
the same formula, replacing $a_\eta(x)$ with $\Es_\eta(x)$.
It is immediate that this is a transient process that never returns
to $\eta$ after time $0$.    The case $\eta = [0]$, $a_\eta = a$
corresponds to random walk conditioned to never return to
the origin.

\begin{lemma}  There exists $c > 0$ such that the following holds.
Let $ 0 < k < n$ and let $\eta$ be a SAW intersecting both $C_k$
and $ \p C_n$.  Let $S$ be a simple random walk starting at $x \in C_k$.
\begin{itemize}
\item  (Beurling)  
\[   \Prob^x\{S[0,\rho_n] \cap \eta = \eset\} \leq c\, e^{(k-n)/2} . \]
\item (Disconnection probability)
\begin{equation}  \label{disconnect}
\Prob^x\{0 \mbox{ is connected to } \p C_n \mbox{ in }
\Z^2 \setminus S[0,\rho_n] \} \leq ce^{(k-n)/4}.
\end{equation}
\end{itemize}
\end{lemma}

We have stated these estimates for random walk starting in $C_k$ ending
at $\p C_n$.  By reversing paths we get analogous statements
for random walks starting in $\p C_n$ stopped upon reaching
$C_k$.

\begin{proof}
The discrete Beurling estimate was first proved by Kesten in \cite{Kesten};
see also, \cite[Theorem 6.8.1]{LL}.   The fact that the disconnection
probability satisfies a power law   up to constants (with no logarithmic correction)
with the Brownian disconnection exponent 
was proved in \cite{LPuckette}.  The value of the exponent
was determined rigorously in \cite{LSWexp}.
\end{proof}

\begad
 The exponents $1/2$ and $1/4$ are known but they
take some effort to prove, especially the latter one.  For our main theorem,  
it would suffice that there is some exponent that satisfies these
conditions and proving that is significantly easier; however,
in order to avoid having extra arbitrary exponents,
we will
use the actual values.
\endad

The following is an easy corollary
of \eqref{potential},
\eqref{sep13.4}, and  the Beurling estimate.

\begin{lemma}  \label{lemma.sep8.0}
There exists $0 < c< \infty$ such that the following
is true.  Suppose $\eta \in \saws_n$.
\begin{itemize}
\item  For $ |z| > e^n$,
\[       a_\eta(z)\geq   \frac{2}{\pi} \, [\log |z| - n ] + O(e^{-n}).\]
\item  For all $z \in C_n$,
\[        a_\eta(z) \leq  c\, e^{-n/2} \,  \left[\dist(z,\eta)\right]^{1/2} .\]
In particular, if $z \in C_k$ with $k < n$,
\[      a_\eta(z) \leq c\, e^{(k-n)/2}.\]
\end{itemize}
\end{lemma}

Another simple idea that we will use is the
following.

\begin{lemma} \label{sep13.lemma1}
There exists $c < \infty$,
such that the following hold if $d=2$ and $n,m \geq 1$.

\begin{itemize}

\item 
Let   $V$ be  a connected
subset of $C_n$ of diameter at least $e^{n}/100$.
If $z \in C_{n+1}$, then the probability that
a random walk starting at $z$ reaches $\p C_{n+m}$
without hitting $V$ is less than $c/m$.

\item Let   $V$ be  a connected
subset of $C_{n+m+1} \setminus C_{n+m}$ of diameter at least $e^{n+m}/100$.
If $z \in \Z^2 \setminus
C_{n+m-1}$, then the probability that
a random walk starting at $z$ reaches $\p C_{n}$
without hitting $V$ is less than $c/m$.

\end{itemize}

\end{lemma}

\begin{proof}  We will do the first; the proof
of the second is similar.
The key fact is that there exist
universal $0 < c_1 < c_2 < \infty$ such that 
\begin{itemize}
\item   
the probability that random walk starting at $z
\in C_{n+1}$ hits
$V$ before reaching $\p C_{n+2}$ is greater than $c_1$ (this can be shown, say, by the invariance principle and
a simple topological argument using planarity);
\item  the probability that a random walk starting
at $w\in \p C_{n+2}$ reaches $\p C_{n+m}$ before
hitting $C_{n+1}$ is less than $c_2/m$. 

\end{itemize}
If $q$ denotes the maximum over $z \in C_{n+1}$
of the probability of reaching $\p C_{n+m}$ before
hitting $V$, we get the  inequality
\[      q \leq (1-c_1) \, [c_2\, m^{-1} + q].\]

\end{proof}

We will also use the following estimate of the
transience of two-dimensional random walk conditioned
to never return to the origin.

\begin{lemma}  \label{lemma.sep8}
Suppose $d=2$ and $\tilde S_j$
is a simple random walk conditioned to never
return to the origin.  For $0 < r \leq 2$,
if $z \in \p C_{n+r}$,
then an $n \rightarrow \infty$,
\[    \Prob^z\{\tilde S[0,\infty) \cap C_n  =  \eset \}
=   \frac r n + O\left(\frac{r }{n^2}\right). \]
Moreover, there exists $c < \infty$
such that  if  $\eta \in \saws_n$,
\[ \Prob^z\{\tilde S[0,\infty) \cap \eta = \eset \}
 \leq \frac{cr}{n}. \]
\end{lemma}

\begin{proof}  From the definition of an $h$-process,
and the fact that the unconditioned walk reaches
$C_n$ with probability one, 
we see that
\begin{equation}  \label{sep13.1}
\frac{\min\{a(x): x \in \p_i C_n\}}{a(z)}
\leq  \Prob^z\{\tilde S[0,\infty) \cap C_n \neq  \eset \}
 \leq   \frac{\max\{a(x): x \in \p_i C_n\}}{a(z)}.
 \end{equation}
Using \eqref{potential} we see
that if  $x \in \p_iC_n$,
\[   a(x) =  \frac{2n}{\pi}  + k_0 + O(e^{-n}) , \]
and
\[   a(z) = \frac{2(n+r)}{\pi} + k_0 + O(e^{-n}).\]
Therefore, both the left and right hand sides of
\eqref{sep13.1}   equal
\[           1 - \frac{r + O(e^{-n})}{n+ r + (\pi k_0/2)  } 
= 1 - \frac{r}{n}  + O\left(\frac{r }{n^2}\right).\]
This gives the first inequality and the second is
done similarly to the previous lemma.
\end{proof}

\subsection{Loop measures}  \label{loopsec}

Here we review some facts about the random walk loop
measure and its relation to LERW; for more details see,
\cite[Chapter 9]{LL}.
We will consider the loop measures  
in $\hat \Z^d:= \Z^d \setminus \{0\}$.  (If we were considering
only the transient case $d \geq 3$, it would be a little easier
to consider the   loop measure in $\Z^d$; however, for $d=2$,
we need to restrict to  $\hat \Z^2$ and this
also works for $d \geq 3$, so we will use this approach.)
A rooted loop is a nearest neighbor path
\[   l = [l_0,l_1,\ldots,l_{2n}] \]
in $\hat \Z^d $ with $n >0$ and $l_0 = l_{2n}$.
The {\em rooted loop measure}  $\tilde m$ is the measure on rooted loops
that assigns measure $[2n(2d)^{2n}]^{-1}$ to each 
loop of $2n$ steps.  An unrooted loop $\ell$ is
an equivalence class of loops under the 
relation
\[  [l_0,l_1,\ldots,l_{2n}] \sim
[l_1,\ldots,l_{2n},l_1] \sim [l_2,l_3,
\ldots,l_{2n},l_1,l_2] \sim \cdots.\]
The {\em (unrooted) loop measure} $m$ is the measure
induced on unrooted loops by the rooted
loop measure
\[       m(\ell) = \sum_{l \in \ell} \tilde m(l).\]
If $B \subset A \subset \hat \Z^d $ , 
we set 
\[   F_B(A) = \exp \left\{\sum_{\ell \subset A, \ell
\cap B \neq \eset }  m(\ell) \right\} =
\exp \left\{\sum_{l \subset A, l
\cap B \neq \eset }  \tilde m(l) \right\} .\]
In other words, $\log F_B(A)$ is the (loop) measure of loops
in $A$ that intersect $B$. 
An equivalent definition (see, e.g., \cite[Propositions 9.3.1, 9.3.2]{LL})
can be given by setting
$B = \{y_1,\ldots,y_m\}, A_k = A \setminus \{y_1,\ldots,y_{k-1}\}$,
in which case
\[   F_B(A) = \prod_{k=1}^{m } G_{A_k}(y_k,y_k), \]
where $G_{A_k}$ denotes the usual random walk Green's
function on $A_k$.
If $B \not\subset A$, we set
$F_B(A) = F_{B \cap A}(A)$.

A {\em loop soup} is a Poissonian realization from
the loop measure.  In particular, if $\loop$ is a set
of loops, then the probability that the loop soup
contains no loop from $\loop$ is 
$   \exp\{-m(\loop)\}.$
When giving measures of sets $\loop$
of loops, one can either give
$m(\loop)$ or one can give the probability of at least one loop,
$1-e^{-m(\loop)}$, and  for small $m(\loop)$
these are equal up to an
error of order  $O(m(\loop)^2)$  (Loops soups with various
intensities are studied in relation to other
models.  In this paper we consider only
the soups with intensity one that corresponds to
the loop-erased walk.)

The next lemma gives a useful way to compute
loop measures of certain sets of loops (see \cite[Section 9.5]{LL}).
If $x \in V$ and we wish to obtain a realization of the loop
soup (with intensity $1$) restricted to loops in $V$ that intersect
$x$ we can do the following:
\begin{itemize}
\item  Start a random walk $S$ at $x$ and stop it at the first  time
$T$ 
that it exits $V$ (for $d \geq 3$, this can be infinity).
\item  Let $\sigma$ be be the last time before $T$ that the
walk is at $x$. This gives a loop $S[0,\sigma]$.  
We can decompose this as a  finite union of   loops that return to
  $x$ only once.
\end{itemize}
 The next lemma follows from this
observation.

\begin{lemma}  \label{sep15.lemma1}
Suppose $A,B$ are disjoint subsets of $\hat \Z^d
.$
Suppose $A$ is finite and the points in $A$  are ordered 
$A = \{x_1,x_2,\ldots,x_n\}$.
Let $A_j
= A \setminus \{x_1,\ldots,x_{j-1}\}$. 
For each $j$, let $S_k$ be a simple random walk
starting at $x_j$ and let \[ T^j = \max\{k: S_k = x_j,
S[0,k]  \cap \{0,x_1,\ldots,x_{j-1}\} = \eset \}.\]  Then
the probability that  a loop soup contains  a loop
in $A_j$ that intersects both $x_j$ and $B$ is 
\[ \Prob^{x_j}\{S[0,T^j] \cap B  \neq  \eset\}.\]
In particular,   the probability that the loop soup
contains a loop  that intersects both   $A$
and $B$ is bounded above by 
\begin{equation}  \label{jun12.2}
\sum_{j=1}^n  \Prob^{x_j}\{S[0,T^j] \cap B   \neq \eset
\}
\end{equation}
\end{lemma}

Note that  $\Prob^{x_j}\{S[0,T^j] \cap B \neq  \eset\}$
is the probability that a simple
random walk starting at $x_j$
reaches $B$ and then returns to $x_j$ without visiting
$\{0,x_1,\ldots,x_{j-1}\}$.  The bound \eqref{jun12.2}
holds regardless of which ordering of the vertices
of  $A$ is used.  The proofs of the following two
lemmas are simply to that in \cite[Lemma 11.3.3]{LL}.

\begin{lemma} \label{june6.lemma1}
There exists
$c = c(d)  < \infty$ such that for all $r,n \geq 0$,
the probability that the loop soup contains
a loop  
that intersects both $C_n$ and $\Z^d
\setminus C_{n+r}$ is bounded above by $c \, e^{r(2-d)}$.
\end{lemma}

\begin{proof}  
We assume $r \geq 2$ and $d \geq 3$ for
the other cases are trivial.  Let $B=\Z^d
\setminus C_{n+r}$.   We write $C_n = \{0,x_1,\ldots,
x_N\}$ where the vertices are ordered so that
$|x_j|$ is nondecreasing.  Let $A_j = \Z^d \setminus
\{x_1,\ldots,x_{j-1}\}$. 
If we start at $x=x_j \in \hat C_n$, then the probability that
it reaches distance $2|x|$ from the origin 
without leaving $A_j$ is $O( |x|^{-1})$; 
the probability that after it leaves $C_{n+r}$
it returns to within distance $2|x|$ of the origin is $O(|x|^{d-2}\, e^{ (n+r)(2-d)})
$;
and given that, the probability of hitting $x$ before any
point in $A_j $  is $O(|x|^{1-d})$.
Hence,
\[  \Prob^{x_j}\{S^j[0,T^j] \cap B \neq \eset\} 
\leq   c \, |x_j|^{-2} \, e^{-(n+r)(d-2)}.\]
Summing over $0 < |x| < e^n$, we get
\[    \sum_{x_j \in C_n}
\Prob^{x_j}\{S^j[0,T^j] \cap B \neq \eset\}  \leq c \, e^{n(d-2)}
\, e^{ (n+r)(2-d) } \leq c\, e^{r(2-d)}.\]

\end{proof}

\begin{lemma} \label{june6.lemma1.2d}
If $d = 2$, there exist 
$c   < \infty$ such that for all $r,n \geq 0$,
the probability that the loop soup contains
a loop  
that intersects both $C_n$ and $\Z^d
\setminus C_{n+r}$  and does not disconnect
$C_n$ from $\p C_{n+r}$
is bounded above by $c \, e^{-r/2}$.
\end{lemma}

\begin{proof}  This is done similarly using \eqref{disconnect}.
We will say just disconnecting for ``disconnecting
$C_n$ from $\p C_{n+r}$''.
Let $B=\Z^2
\setminus C_{n+r}$, and write  $C_n = \{0,x_1,\ldots,
x_N\}$ where the vertices are ordered so that
$|x_j|$ is nondecreasing.  Let $A_j = \Z^d \setminus
\{x_0,\ldots,x_{j-1}\}$. 
If we start at $x=x_j \in C_n$, then the probability that
it reaches distance $2|x|$ from the origin 
without leaving $A_j$ is $O( |x|^{-1})$; given that, the probability
that it reaches $\p C_n$ without leaving $A_j$
is $O((n+1 - \log |x|)^{-1})$; given this,   the probability that
it reaches $\p C_{n+r}$ without disconnecting 
is $O(e^{-r/4})$; given this,
the probability that after it leaves $C_{n+r}$
it returns to $C_n$  without disconnecting   is $O(e^{-r/4})$; given this, the
probability to get within $2|x|$ without disconnecting
 is $O((n+1 - \log |x|)^{-1})$;
and given that, the probability of hitting $x$ before any
point in $A_j $  is $O(|x|^{-1})$.
Hence,
\[  \Prob^{x_j}\{S^j[0,T^j] \cap B \neq \eset, \mbox{no disconnection}\} 
\leq   c \, |x_j|^{-2} \, [n+1-\log |x|]^{-2} \,  e^{-r/2}.\]
Summing over $0 < |x| < e^n$, we get
\[    \sum_{x_j \in C_n}
\Prob^{x_j}\{S^j[0,T^j] \cap B \neq \eset, \mbox{no disconnection}\}  \leq   c\, e^{-r/2}.\]
\end{proof}


\begin{lemma} 
If $d \geq 3$, let  $\loop_n(\delta)$ denote the set of loops that
intersect $C_{n+2} \setminus C_{n-2}$ and have
diameter at least $\delta \, e^n$.  
For every $\delta > 0$, there exists $\epsilon > 0$, such
that for every $n$, the probability that the loop
soup contains
no loop  in $\loop_n(\delta)$  is 
at least $\epsilon$.
\end{lemma}

\begin{proof}  This follows from   Lemma \ref{june6.lemma1} and a simple
covering argument. \end{proof}

\begin{lemma}  There exists $c < \infty$ such that
the following holds.  Suppose
$d = 2$, $n,m$ are positive integers, and
$V$ is a simply connected subset of $\Z^2$ with
\[   C_{n+m} \subset V, \;\;\;\; C_{n+m+1} \not \subset V.\]
Let  $K = K_{n,V}$ denotes 
the measure of the set of loops
that lie in $\Z^2 \setminus C_{n-1}$ and intersect both $C_n$
and $\p  V$.  Then, 
\[        \left|K  - \frac 1m \right| \leq 
c \,  \frac{me^{-n} + 1} {m^2}  .\]
\end{lemma}

\begin{proof}

We first consider the case $V = C_{n+m}$.
We order the elements of $\Z^2 = \{x_0 = 0,x_1,x_2,\ldots\}$
so that $|x_0| \leq |x_1| \leq |x_2| \leq \cdots.$    We let $A_k
= \{0,x_1,\ldots,x_k\}$ and define $k_n $ by
$C_{n} = A_{k_n }$ (so that $k_n \sim \pi \, e^{2n}$).
Let $
\theta_k = \pi \, \Cp (A_k)/2$, and $g_k = \pi \, a_{A_k}/2$ 
which is the unique function
that is $0$ on $A_k$, discrete harmonic on $\Z^2 \setminus
A_k$, and satisfies
\[   g_k(z) =   \log |z| -
    \theta_k + o(1) , \;\;\;\; |z| \rightarrow \infty.\]
   If $S_t$ denotes a simple random walk and
   \[  \rho_n = \min\{t: |S_t| \geq e^n\}, \]
then
\[          g_k(z) = \lim_{r \rightarrow \infty}
   r\, \Prob^z\{S[0,\rho_r]
     \cap A_k = \eset \}. \]
     Using well-known estimates (see
     \cite[Proposition 6.4.1]{LL}), we see that
for $z \in C_{n+m+1} \setminus C_{n+m}$, $k_{n-1} \leq k \leq k_n$, 
\[     \lim_{r \rightarrow \infty}
      r \, \Prob^z\{ S[0,\rho_{r}] \cap A_{k-1} = \eset\} =
            m + O(1) . \]
If we write $\Prob,\E$ for $\Prob^{x_k},\E^{x_k}$, we have  
\begin{eqnarray*}
\lefteqn{ g_{k-1}(x_k )}  \\
& = & \lim_{r \rightarrow \infty}
      r \, \Prob \{S[0,\rho_r] \cap A_{k-1} = \eset\} \\
& = & \lim_{r \rightarrow \infty}
      r \, \Prob \{S[0,\rho_{n+m}] \cap A_{k-1} = \eset\}
       \, \Prob  \{S[0,\rho_r] \cap A_{k-1} = \eset
        \mid S[0,\rho_{n+m}] \cap A_{k-1} = \eset\}\\
       & = & \Prob\{S[0,\rho_{n+m}]
         \cap A_{k-1} = \eset\} \, [m + O(1)], 
      \end{eqnarray*}
and hence
\[ \Prob \{S[0,\rho_{n+m}]
         \cap A_{k-1} = \eset\} = \frac{g_{k-1}(x_k)}
         {m} \, \left[1 + O(m^{-1})\right].\]

Note that
\[    g_k(z) - g_{k-1}(z)   = -\hm_{A_{k}}(z,x_k) 
 \, g_{k-1}(x_k) , \]
 where $\hm$ denotes harmonic measure, that is, the
 hitting distribution of $A_k$ starting at $z$.
 We will use the estimate (this follows from
 \cite[Proposition 6.4.5]{LL}),
 \[ 
\hm_{A_k}(z,x_k)  =
    \hm_{A_k}(\infty,x_k) \, \left[1 +
     O\left(me^{-m}  \right)
     \right].\]
(We believe the error is actually $O(e^{-m})$ 
but it would take a little more effort to prove and we
do not need the stronger result.)
Therefore,
\begin{eqnarray*}
\theta_k - \theta_{k-1} & = & \lim_{z \rightarrow
\infty} \left[ g_{k-1}(z)-g_k(z)\right]\\
& = & \hm_{A_k}(\infty,x_k)
 \, g_{k-1}(x_k)\\
 & = & \hm_{A_k}(z,x_k)
 \, g_{k-1}(x_k)\,\left[1 +
     O\left(me^{-m} \right)
     \right].
     \end{eqnarray*}

 Using Lemma \ref{sep15.lemma1} we see that
  the  probability that the loop soup contains a loop including $x_k$,
 lying in $\Z^2 \setminus A_{k-1}$, and also intersecting $\p C_{n+m}$
 is   equal to
 \begin{eqnarray*}
 \lefteqn{ \E^{x_k}\left[\hm_{A_k}(S_{\rho_{n+m}},x_k);
    S[0,\rho_{n+m}] \cap A_{k-1} = \eset\right]}
     \hspace{1in} \\
 & = & \Prob\{
   S[0,\rho_{n+m}] \cap A_{k-1} = \eset\}
   \, \E^{x_k}\left[\hm_{A_k}(S_{\rho_{n+m}},x_k)
     \mid  S[0,\rho_{n+m}] \cap A_{k-1} = \eset\right]\\
   & = & \frac{g_{k-1}(x_k)}{m} \, \frac{\theta_k - \theta_{k-1}}{g_{k-1}(x_k)} \,  [1 + O(m^{-1})]\\
   & = & \frac{\theta_k - \theta_{k-1}}m \, 
    [1 + O(m^{-1})].
 \end{eqnarray*}
%
It follows that the
measure of the set of loops that lie in $\Z^2 \setminus A_{k-1}$,
contain $x_k$, and intersect $\p C_{n+m}$ is
\[     \frac{\theta_k - \theta_{k-1}}{m}   \,
   [1 + O(m^{-1} )] .\]
The capacities of 
$C_{n-1}$ and $C_n$ are well known up
 to a small error
(see \cite[Proposition 6.6.5]{LL}); indeed,
\[    \theta_{k_n} =  \theta_{k_{n-1}} +  1  + O(e^{-n}), \]
and hence
\[ \sum_{j=k_{n-1} + 1}^{n_k} \E^{x_j}\left[\hm_{A_k}(S_{\rho_{n+m}},x_k);
    S[0,\rho_{n+m}] \cap A_{k-1} = \eset\right]
       = \frac{1 + O(e^{-n})}{  m} + O\left(\frac{1 }{m^2}\right).\]

For more general $V$, we use the fact that $V$
is simply connected and $\p V \cap C_{n+m+1} \neq
\eset$ to see that  the probability that a random walk
starting in $\p C_{n+m}$ reaches $\p C_{n+1}$
without hitting $\p V$ is $O(m^{-1})$ (see Lemma
\ref{sep13.lemma1}).  Arguing as
above, we can see that 
the
measure of the set of loops that lie in $\Z^2 \setminus A_{k-1}$,
contain $x_k$,  intersect $\p C_{n+m}$, but do
{\em not} intersect $\p V$ is $O(m^{-2})$. 
       \end{proof}

\begin{lemma}   \label{sep14.lemma1}  Suppose $d = 2$.  There exists
$c<\infty$ such that the following holds.
\begin{itemize}
\item  Let $A$ be a simply connected subset of
$\Z^2$ with $e^{n+1} \leq \dist(0,\p A) \leq
 e^{n+1} + 1$, and let $L = L_A$ denote the measure
 of loops in $\hat \Z^2$ that intersect both $C_n$
 and $\p A$.  Then,
 \[   |L-\log n| \leq c .\]

\item  For every $\delta >0$, there exists $c_\delta
< \infty$ such that the measure of loops in $C_{n+1}$
that intersect $C_{n+1} \setminus C_n$; are of
diameter at least $\delta \, e^{n}$; and do not
disconnect $0$ from $\p C_{n+1}$ is bounded above
by $c_\delta$. 
\end{itemize}
\end{lemma}

\begin{proof}  $\;$
\begin{itemize}
\item This  follows from
the previous lemma by summing. 

\item  The measure of the set of loops
in  $\Z^2 \setminus C_{n-j}$
that intersect both $C_{n-j+1}$ and $\Z^2 \setminus
C_{n+1}$  and do not disconnect $0$ from
$\p C_n$ is $O(e^{-j/4})$.
\end{itemize}

\end{proof}

%

\subsection{A lemma about simple random walk}  \label{rwlemmasec}

Here we discuss a lemma about simple random walk that plays a crucial
role in our analysis.  It is very believable, but the important fact
is that a constant can be chosen uniformly. We first state the
result and gives some important corollaries.  The $d=2$ case
was done in \cite[Propositon 3.5]{Masson} and the $d = 3$
case was  proved in \cite{Shir}.  For completeness, 
we discuss the
proof in the appendix.  Here
$S_j$ denotes a simple random walk and $\Prob^x,\E^x$ denote
probabilities and expectations assuming that $S_0 = x$. 

\begin{lemma}  \label{importantlemma}
There exists $c > 0$ such that the following is true.
\begin{enumerate}
\item  Suppose $A' \subset C_n, z \in \p C_n, A = A'
\cup \{z\}$.  Let
$\tau = \tau_A = \min\{j \geq 1: S_j \in A\}$ and $\sigma_r =
\min\{j: |S_j - S_0| \geq r\}.$   Then,
\begin{equation}  \label{jun11.1}
\Prob^z\left \{|S_{\sigma_r}| \geq e^n + \frac r2
\mid \sigma_r < \tau \right\}  \geq c . 
\end{equation}

\item  Suppose $r < e^n/2$, $A' \subset \Z^d \setminus
C_n, z \in\p_i C_n$, $A = A' \cup \{z\}$.  Then,
\begin{equation}  \label{jun11.2}   \Prob^z\left \{|S_{\sigma_r}| \leq e^n - \frac r2
\mid \sigma_r < \tau \right\}  \geq c . 
 \end{equation}

\end{enumerate}
\end{lemma}

The proof strongly uses the fact that $A'$ is in $C_n$ (or
$\Z^d \setminus C_n$) and the random walk is starting on the
boundary of $C_n$.  We will discuss the proof of
\eqref{jun11.2}  in Section \ref{rwproofsec} (this is the harder
case), 
but we give some preliminary reductions here.
\begin{itemize}
\item 
It suffices to prove the lemma for $n$ sufficiently
large, for then the small $n$ can be done on a case by case basis.

\item Using the invariance principle, it suffices
to consider $r \leq \delta_0 \, e^n$ for some $\delta_0 > 0$.

\item Using the invariance principle, it suffices
to establish  \eqref{jun11.1} and 
\eqref{jun11.2}  with $\frac r2$ replaced with $\epsilon r$
for some $\epsilon >0$.

\item 
It suffices to find $c_1$ such that for $n$ sufficiently
large and $r  \leq \delta_0 \, e^n$,  \eqref{jun11.1} and 
\eqref{jun11.2} hold for some  $z_1$ with $|z-z_1|
\leq c_1$.  Indeed, there is a positive probability (bounded
uniformly from below) that a random walk starting at $z$ reaches
$z_1$ without visiting $A$.
\end{itemize}
We will derive a number of corollaries of this lemma.

\begin{corollary}  \label{cor1}
Suppose $d  \geq 3$. 
There exists $c > 0$ such that if  we choose
$r = e^{n-4}$
in part 1 of Lemma \ref{importantlemma}, then
\[        \Es_A(z) \geq c \, \Prob^z\{\sigma_r < \tau\}
. \]
\begin{itemize}
\item In particular, if $A_1',A_2' \subset C_n$
agree in the disk of radius $e^{n-4}$ about $z$, then
\[   \Es_{A_1}(z) \asymp \Es_{A_2}(z) . \]
\item  If  
$A_1' , A_2' \subset C_n$   and
$A_1' \cap (C_n \setminus C_{n-j}) = A_2' \cap
(C_n \setminus C_{n-j})$ for some $j \geq 1$,
then
 \[   \Es_{A_1}(z)= \Es_{A_2}(z)\,
 [1 + O(e^{j(2-d)})].   \]
 \end{itemize}
\end{corollary}

\begin{proof}  We write $\Prob$ for $\Prob^z$.
The lemma tells us that 
\[   \Prob\{|S_{\sigma_r}| \geq e^n +  e^{n-5}
\mid \sigma_r < \tau \} \geq c . \]
Since $d \geq 3$,  
there exists $c_1$ such that the probability that
a random walk starting at distance $e^n + e^{n-5}$
from the origin
never returns to $C_n$ is greater than $c_1$.   Hence,
there exists $c_2 < \infty$ such that
\[  \Prob\{S[1,\infty) \cap A = \eset \mid \sigma_r < \tau\}
\geq c_2. \]

For the final bullet   note that $|\Es_{A_1}(z) - \Es_{A_2}(z)|$
is bounded above by  $\Prob\{\sigma_r < \tau\}$ times
the conditional probability given this that the random walk enters
$C_{n-j}$.  The latter probability is $O(e^{j(2-d)})$, and hence
\[ |\Es_{A_1}(z) - \Es_{A_2}(z)| \leq c \, e^{j(2-d)}
\, \Prob\{\sigma_r < \tau\} \asymp  e^{j(2-d)} \, \Es_{A_1}(z).\] 
\end{proof}

We will give a similar result for $d=2$, but we will
put in an additional condition.  We say that $z$
is connected to $0$ in $A$ if there exists
a SAW $\eta \in \pairs_n$ with terminal vertex
$z$ with $\eta \subset A$.  We similarly can say
that $z$ is connected to infinity in an infinite
set $A$.

\begin{corollary}  \label{cor1.2d}
Suppose $d = 2$.
There exists $c > 0$ such that if  we choose
$r = e^{n-4}$
in part 1 of Lemma \ref{importantlemma}, then
\[        \Es_A(z) \geq c \, n^{-1} \,  \Prob^z\{\sigma_r < \tau\}
.\]
Moreover, if 
$z$ is connected to $0$ in  $A_1,A_2$, the following hold.
\begin{itemize}
\item If $A_1',A_2' \subset C_n$
agree in the disk of radius $e^{n-4}$ about $z$, then
\[   \Es_{A_1}(z) \asymp \Es_{A_2}(z)
\asymp   n^{-1} \,  \Prob^z\{\sigma_r < \tau\}. \]
\item   If
$A_1' \cap (C_n \setminus C_{n-j}) = A_2' \cap
(C_n \setminus C_{n-j})$ for some $j \geq 1$,
then
 \[   \Es_{A_1}(z)= \Es_{A_2}(z)\,
 [1 + O(e^{-j})], \;\;\;\; d =2,   \]
 \end{itemize}
\end{corollary} 

\begin{proof}  The proof is similar.  For
the first inequality 
we use Lemma \ref{lemma.sep8} to see
that
\[  \Prob\{\tilde S[1,\infty) \cap A = \eset \mid \sigma_r < \tau\}
\geq c_2\, n^{-1} , \]
where $\tilde S$  is random walk conditioned to avoid
the origin.

If  $z$ is connected to $0$
in $A_1$, we can see  
from the Harnack inequality, Lemma \ref{sep13.lemma1},
and the   Beurling 
estimate that there exists uniform $0 < c_3 < c_4
< \infty$ such that
\[   c_3 \leq a_{A_1}(w) \leq c_4, \;\;\; w \in \p C_{(1+r)n},\]
\[     a_{A_1}(w)  \leq c_4\, e^{-j/2}, \;\;\; w \in C_{n-j} , \]
\[    a_{A_1}(w) \geq c_3, \;\;\; w \not\in C_{(1+r)n}.\]
Also, the probability that a random walk
starting at $z$
reaches $C_{n-j}$ without returning to  $A$ is bounded
above by $O(e^{-j/2})$;  if it succeeds in doing this,
there is at most a  $O(e^{-j/2})$ probability
that it returns to $\p C_n$ without hitting $A$.  
Hence, conditioned that a random walk avoids
$A_1$, the probability that it  hits   $C_{n-j}$
is $O(e^{-j})$ which implies that
\[    \Es_{A_1 \cup C_{n-j}}(z) \geq
   \Es_{A_1}(z) \, [1-O(e^{-j})], \]
   and similarly for $A_2$.

\end{proof}

The following was given in the proofs but it is important enough to state
it separately.

\begin{corollary}  \label{feb7.cor1}
  There exists $c < \infty$ such that if $\eta \in 
\saws_n$ with terminal point $z$, then the probability that a simple
random walk starting at $z$ conditioned to never return to $\eta$
enters $C_{n-j}$ is less than $c e^{-j}$.
\end{corollary}

\begin{corollary}  \label{cor2}
There exists $c > 0$ such that if  
we choose $r = e^{n-4}$
in part 2 of Lemma \ref{importantlemma}, and $B =
\Z^d \setminus A$, then for $d \geq 3$,
\[   H_{B }(0,z) \asymp  \, e^{n(2-d)} \, \Prob^z\{\sigma_r < \tau\} .\]
If $d=2$ and $A$ contains a connected path of diameter $e^{n-4}$ including $z$, then the same result is true.
\end{corollary} 

\begin{proof}  
By reversing paths, we see that
\[  H_{B }(0,z) =   \E^z\left[G_B(0,S_{\sigma_r}) ; \sigma_r < \tau \right].\]
For the upper bound, we use $G_B(0,S_{\sigma_r}) \leq  
c\,e^{n(2-d)}$, which for $d=2$ requires the extra
assumption.  For the lower bound, we use
\[     \Prob^z \{|S_{\sigma_r}| \leq e^n - e^{n-5} \mid \sigma_r < \tau\}
\geq c , \]
and for $|x| \leq e^n - e^{n-5}$,
\[         G_B(0,x) \geq G_{C_n}(0,x)  \geq c \, |x|^{ 2-d }.\]
\end{proof}

\begin{corollary}  \label{cor3}
If  $n \leq m-1$,
$\eta  \in \saws_n$ with terminal point $y$, and $\bar \eta 
\in \barsaws_{n+1,m}$ with initial point $w$, 
then if $A = \Z^d \setminus (\eta  \cup \bar \eta )$,
\[      H_{\p A}(y,w) \asymp
\left\{\begin{array}{ll}
\Es_\eta(y) \, H_{\Z^d \setminus\bar \eta }
(0,w), & d \geq 3\\
n\, \Es_\eta(y) \, H_{\Z^d \setminus\bar \eta }
(0,w), & d =2 \end{array} \right. .\]
\end{corollary}

\begin{proof}  Let $S,\tilde S$ be independent random walks
starting at $y,w$, and let $\sigma_r,\tilde \sigma_r$ be
the corresponding stopping times with $r = e^{n-4}$.  Any random walk
path $\omega$ from $y$ to $w$ in $A$ can be decomposed as
\[   \omega = \omega^- \oplus \tilde \omega \oplus \omega^+, \]
where $\omega^-$ is the walk stopped at the first time it reaches
distance $r$ from $y$, and $\omega^+$ is the reversal
of the reversed walk stopped at the first time it reaches distance
$r$ from $z$.
Using this decomposition, we can see that for $d \geq 3$,
\begin{eqnarray*}
H_{\p A}(y,w) &  =  & \sum_{x,z}
\Prob^y\{S_{\sigma_r} = x; \sigma_r < \tau\}
 \,  \Prob^w\{S_{\tilde \sigma_r} = z; \tilde\sigma_r <
 \tilde \tau\}\, G_A(x,z)\\
& \asymp & \Es_\eta(y) \, e^{n(2-d)} \,\Prob^w\{\tilde \sigma_r
< \tilde \tau\} \\
  & \asymp &  \Es_\eta(y) \,  H_{\Z^d \setminus\bar \eta }
(0,w).
\end{eqnarray*}
For $d=2$, we need to replace $\Es_\eta(y)$ with
$n  \, \Es_\eta(y)$.
\end{proof}

\begin{corollary}  \label{cor4.1}
If  $n \leq m-1$,
$\eta,\tilde \eta  \in \saws_n$ with terminal point $y$
and such that $\eta \setminus C_{n-j} = \tilde \eta
\setminus C_{n-j}$, and $\bar \eta 
\in \barsaws_{n+1,m}$ with initial point $w$, 
then if $A = \Z^d \setminus (\eta  \cup \bar \eta )$,
$\tilde A =  \Z^d \setminus (\tilde \eta  \cup \bar \eta )$,
\[      H_{\p A}(y,w)  = H_{\p \tilde A}
(y,w) \, [1 + O(e^{-j})]. \]
%
\end{corollary}

\begin{proof}
We start as in the last proof with
\[  H_{\p A}(y,w)    =    \sum_{x,z}
\Prob^y\{S_{\sigma_r} = x; \sigma_r < \tau\}
 \,  \Prob^w\{S_{\tilde \sigma_r} = z; \tilde\sigma_r <
 \tilde \tau\}\, G_A(x,z), \]
 and similarly for $\tilde A$ and then use
 \[       G_A(x,z) = G_{\tilde A}(x,z)
   \, [1 + O(e^{-j})].\] 
For $d=2$, this uses the Beurling estimate.
 \end{proof}

There is a simple fact about the loop-erasing process that we will use.
We state it as a proposition (so we can refer to it), but it is an easily verified
property of the deterministic loop-erasing procedure.

\begin{proposition}  \label{looperasefact}
Suppose $S$ is a simple random walk starting at $x \in C_n$ and
$n <  k < m$.  Suppose that 
\begin{itemize}
\item After the first visit to $\p C_k$, the   walk never returns
to $C_n$.
\item After the first visit to $\p C_m$, the  walk never returns
to $C_k$.
\end{itemize}
Suppose that we stop the path some time after it reaches  $\p C_m$
and erase loops.  After doing this, we view the remainder of
the random walk and continue loop-erasing (and hence perhaps
erasing some of the loop-erasure we already have).  Then
\begin{itemize}
\item  The intersection of the original and the new
loop-erased paths with $C_n$ are the same.
\end{itemize}
\end{proposition}

Indeed, in order to erase a point $x$ in the intersection of
the loop-erasure and $C_n$, the random walk would have to visit
a point on the random walk that was visited before the last
return to $x$.  There is no point in $\Z^d \setminus C_k$ that
satisfies this, and the random walk visits no point in $C_k$
after it has reached $\p C_m$.

This gives a general procedure to give lower bounds
on the probabilities of certain    events
for the loop-erased walk.
\begin{itemize}
\item  If $n < m$, then in the measure $\mu_m$, given the
initial segment $\eta_n \in \saws_n$, the continuation is obtained
by taking a simple random walk conditioned to avoid $\eta_n$ and 
erasing loops.

\item  Using our lemma and its corollaries, there is a positive
probability that this simple random walk will start by  reaching
radius $e^{n + (1/10)}$ without going more than distance
$e^{n + (1/5)}$ from the starting point.

\item  Given that, we can consider random walk conditioned to
avoid $C_n$ which is conditioning on an event of positive probability
uniformly bounded away from zero.  For $d=2$,
we need to use random walk conditioned to avoid some
$\eta \in C_n$ (see Lemmas \ref{lemma.sep8.0}
and \ref{lemma.sep8}).

\item  The loop-erased path is a subpath of the simple random walk
path, so if we know the simple path stays in some set, then so
does the loop-erased path.

\end{itemize}

There are many applications of this; we state one as a corollary here.

\begin{corollary} \label{extendcor}
There exists $c > 0$ such that the following holds.  Suppose
$\eta  \in \saws_n$ with terminal point $y$ and $\bar \eta 
\in \barsaws_{n+1,m}$ with initial point $w$, 
$A = \Z^d \setminus (\eta  \cup \bar \eta )$, and $\omega$ is
a simple random walk excursion
starting at $y$ conditioned to leave
$A$ at $w$.  Let
\[     V =  V_n(y,w)= (C_{n+1} \setminus C_n)
\cup \{w' :|w' - w| \leq e^{n-3}\} \cup \{y': |y - y'|
\leq e^{n-3} \}. \] Then,
\begin{itemize}
\item  The probability that $\omega \subset V$
is at least $c$.
\item  If it is also known that
$         y,w \in \{x =(x_1,x_2,x_3) \in \Z^3:
x_1 \geq |x|/2 \}, $
then the probability that 
$ \omega  \subset
\{x =(x_1,x_2,x_3) \in V:
x_1 \geq |x|/4 \}$
is at least $c$.
\end{itemize}
\end{corollary}

Obviously these results hold for the loop-erasure of
$\omega$ as well.

\subsection{Loop-erased random walk}  \label{lerwsec}

In this section  discuss facts about
a single loop-erased random walk (LERW) in
$\Z^d, d \geq 2$.

\begin{itemize}

\item If $S_j$ is a simple random walk starting at
the origin  conditioned to
never return to the origin  with loop-erasure
$\hat S_j$, we let
\[       \rho_n = \min\{j: S_j \not \in C_n\} ,\;\;\;\;
\hat \rho_n = \max\{j: S_j \in C_n\} , \]
\[  T_n = \min\{j: \hat S_j \not \in C_n\},\;\;\;\;
\bar T_n = \max\{j: \hat S _j \in C_n\} ,\]
\[ \bar T_{n,m} = \max\{ j \leq T_m:\hat S_j \in C_n\}.\]
Note that $\bar T_n + 1 \geq T_n$.

\item $\mu_n$ is the distribution of $\hat S[0,T_n]$.  It is
a probability measure supported on $\saws_n$.


\item  If $n < m$, $\mu_{n,m}$ is the distribution of
$\hat S  [\bar T _{n,m},T _m]$.  It is a probability measure
on $\barsaws_{n,m}.$

\item  In the next subsection we
also write $\mu_n$ and $\mu_{n,m}$ for the product measures
$\mu_n \times \mu_n$ on $\saws_n^2$ and $\mu_{n,m}
\times \mu_{n,m}$ on  $\barsaws_{n,m}^2.$

\end{itemize}

We recall the following fact that
follows from the decomposition \eqref{jun24.1}.  Let
\[    G_0 = \left\{\begin{array}{ll} G(0,0), & d \geq 3 \\
    1, & d  = 2 . \end{array} \right. \]

\begin{proposition}  \label{basiclerw}
If $\eta \in \saws_n$ with terminal point $z$,
then
\begin{equation}  \label{jun3.1}
\mu_n (\eta)
=  (2d)^{-|\eta|} \, F_\eta \, G_0 \, \Es_\eta(z).  
\end{equation}
Moreover, the distribution of $\hat S[T_n,\infty]$
given $\hat S[0,T_n] = \eta$ is the same as that
obtained as follows:
\begin{itemize}
\item Take
a simple
random walk starting at $z$ conditioned to never return
to $\eta$.
\item Erase the loops chronologically.
\end{itemize}
\end{proposition} 

\begad

As we have mentioned, there are two different ways to define ``loop-erased
random walk stopped at $\p C_n$'': one is as the loop-erasure of $S[0,\rho_n]$,
and the other is as $\hat S[0,T_n]$.  These measures are significantly
different near the terminal point.  However,  considered them as measures
on $\saws_{n-1}$ by truncation, they  are comparable.

We prefer to consider $\mu_n$, that is the distribution of $\hat S[0,T_n]$, because
we know that the distribution of the remainder of the path can be obtained by
erasing loops from a simple random walk starting at $\hat S(T_n)$ conditioned
to never return to $\hat S[0,T_n]$.  The estimates from Section \ref{rwlemmasec}
apply to the conditioned random walk and the loop-erasure is a subpath of the
conditioned walk.
As an example, Corollary \ref{feb7.cor1} implies that for $d=2,3$,
 there exists   $c < \infty$
such that conditioned on $\hat S[0,T_n]$, the probability that $\hat S[T_n,\infty)$
intersects $C_{n-j}$ is less than $c \, e^{-j}.$ 

\endad

If $d=2$, we define $\kappa_n$ by saying that $\log \kappa_n$
is the measure of loops in $\hat \Z^2$ that disconnect $0$
from $\p C_n$.  (We do not require the loop to lie in $\hat C_n$.)
In this case, if $\eta \in \saws_n$, we can write
\begin{equation} \label{jun3.1.alt}
\mu_n (\eta)
=  (2d)^{-|\eta|} \, \kappa_n \,  F^{*,n}_\eta \, \Es_\eta(z), \;\;\;\;d=2,
\end{equation}
where $\log F^{*,n}_\eta$ is the measure of loops in $\hat \Z^2$
that intersect $\eta$ but do not disconnect $0$ from $\p C_n$.


The distribution $\mu_{n,m}$ is a little complicated, but we
will only need to know it up to uniform multiplicative constants.

\begin{proposition}
If $n  \leq m-1$, and $\eta \in \barsaws_{n ,m}$
with initial point $w$ and terminal point $z$, then
\[  \mu_{n ,m}(\eta) \asymp 
(2d)^{-|\eta|} \, F_\eta \,\Es_\eta(z)\,
H_{\p(\hat \Z^d \setminus \eta)}(0,w).\]
\end{proposition}

Recall that $H_{\Z^d \setminus \eta}(0,w) =G_{\Z^d \setminus \eta}(0,0)\,
H_{\p(\hat \Z^d \setminus \eta)}(0,w)
 $.
If $d \geq 3$, $G_{\Z^d \setminus \eta}(0,0) \asymp 1$.
  If $d=2$,  
$G_{\Z^d \setminus \eta}(0,0) \asymp n$.

\begin{proof}   We start by writing the exact expression
\begin{eqnarray*}
\mu_{n ,m}(\eta) & =  & \sum_{\eta' \oplus \eta \in \saws_m}
      \mu_m(\eta' \oplus \eta) \\
     & = & (2d)^{-|\eta|} \, F_\eta \,
           \sum_{\eta' \oplus \eta \in \saws_m} \Es_{\eta' \oplus
           \eta}(z) \, (2d)^{-|\eta'|} \, F_{\eta'}(\hat \Z^d \setminus
           \eta). 
\end{eqnarray*}

For the  upper bound, we use $\Es_{\eta' \oplus
           \eta}(z) \leq \Es_\eta(z)$ to see that
\[ \mu_{n ,m}(\eta)  \leq (2d)^{-|\eta|} \, F_\eta \,\Es_\eta(z)
\,  \sum_{\eta' \oplus \eta \in \saws_m}  
 (2d)^{-|\eta'|} \, F_{\eta'}(\hat \Z^d \setminus
           \eta). \]
The last sum is larger if we remove the restriction that $\eta' \oplus
\eta \subset C_m$ and write just
\begin{equation}  \label{jun12.5}
\sum_{\eta'   }  
 (2d)^{-|\eta'|} \, F_{\eta'}(\hat \Z^d \setminus
           \eta), 
           \end{equation}
where the sum is over all SAWs $\eta'$ starting at the origin,
ending at $w$, and otherwise staying in $\Z^d \setminus \eta$.
Using a decomposition similar to \eqref{jun24.1}, we can see that
$(2d)^{-|\eta'|} \, F_{\eta'}(\Z^d \setminus
           \eta)$
 is exactly the probability that a random walk starting at $0$ 
 stopped upon reaching $\eta \cup \{0\}$ 
 stops in finite time
  and the loop-erasure of the stopped walk is $\eta'$. 
Therefore the sum in \eqref{jun12.5} equals 
$H_{\p(\hat \Z^d \setminus \eta)}(0,w)$.

For the lower bound we write
\[ \mu_{n ,m}(\eta)  \geq (2d)^{-|\eta|} \, F_\eta \,
           \sum_{\eta' \oplus \eta \in \saws_m,\;\; \eta'\subset
           C_{n+ \frac 12} } \Es_{\eta' \oplus
           \eta}(z) \, (2d)^{-|\eta'|} \, F_{\eta'}(\hat \Z^d \setminus
           \eta). \]
Using Corollaries \ref{cor1} and \ref{cor1.2d},
  we can see this is greater than a
constant times
\[ (2d)^{-|\eta|} \, F_\eta \,\Es_\eta(z)\, 
           \sum_{\eta' \oplus \eta \in \saws_m,\;\; \eta'\subset
           C_{n+ \frac 12} }   (2d)^{-|\eta'|} \, F_{\eta'}(\hat \Z^d \setminus
           \eta). \]
As in the last paragraph, we see that 
the  sum equals $H_{\p(\hat C_{n+ \frac 12}  \setminus \eta)}(0,z)$.
We can use Corollary \ref{cor2} to see that
\[  H_{\p( \hat C_{n+ \frac 12}  \setminus \eta)}(0,z)
\geq c \,   H_{\p(\hat \Z^d \setminus \eta)}(0,z).\]

\end{proof}

We will now focus on the decomposition \eqref{decomp2} of $\eta \in \saws_m$.
The next lemma shows that $\eta_n$ and $\eta_{n+1,m}$ are ``independent up
to multiplicative constant''.  A two-dimensional version of this result
can be found in \cite[Section 4.1]{Masson}.

\begin{proposition} \label{indprop}
If $n \leq m-2$, $\eta \in \saws_n,\tilde \eta \in \saws_{n+1,m}$, 
then
\begin{equation}  \label{jun24.2}
\sum_{\eta^*} \mu_m(\eta \oplus \eta^* \oplus
\tilde  \eta )\asymp \mu_n(\eta) \, \mu_{n+1,m}(\tilde \eta ). 
\end{equation}
Here the sum is over all SAWs $\eta^*$ such that
$\eta \oplus \eta^* \oplus
\tilde \eta \in \saws_m$.

\end{proposition}

\begin{proof}  We will write $\eta' = \eta \oplus \eta^* \oplus
\tilde \eta$. 
Let $y$ be the terminal point of $\eta$, and let
$w,z$ be the initial and terminal points of $\eta'$, respectively.
Let $A = \Z^d \setminus (\eta \cup \tilde \eta)$, and note
that
\[   F_{\eta \oplus \eta^* \oplus
\tilde \eta } = F_{\eta} \, F_{\tilde \eta}( \Z^d \setminus \eta) \,
F_{\eta^*} (A) . \]
Since $\eta \subset C_n \cup \p C_n$ and $\tilde \eta\subset (\Z^d
\setminus  C_{n+1} ) \cup 
\p_i C_{n+1}$, it follows from Lemma \ref{june6.lemma1} that
$  F_{\tilde \eta} \asymp F_{\tilde \eta}( \Z^d \setminus \eta)$
for $d \geq 3$, and by Lemma \ref{sep14.lemma1}  we see that 
$ F_{\tilde \eta}^{*,n} 
\asymp F_{\tilde \eta} ( \Z^d \setminus \eta)$ for $d=2$.
Therefore the sum on the left-hand side of \eqref{jun24.2}
is comparable to
\begin{equation}  \label{jun12.6}
F_{\eta} \,F_{\tilde \eta} (2d)^{-|\eta|-|\tilde \eta|}
\sum_{\eta^*} (2d)^{-|\eta^*|} \,  F_{ \eta^* } (A)\,
\Es_{ \eta'}(z) , \;\;\;\; d \geq 3,
\end{equation}
and similarly for $d=2$ with $F_{\tilde \eta}$ replaced
with $F_{\tilde \eta}^{*,n}$.

For an upper bound, we use   $\Es_{ \eta'}(z) \leq \Es_{\tilde \eta}(z)$
to bound the sum by
\[  \Es_{\tilde \eta}(z) \sum_{\eta^*} (2d)^{-|\eta^*|} \,  F_{\eta^*} (A),\]
where the sum is over all SAWs from $y$ to $w$ and otherwise in
$A$.  The sum therefore equals 
$ H_{\p A}(y,w) .$

For the lower bound we restrict the sum in \eqref{jun12.6}
to $\eta^*$ such that
$\eta^* \subset C_{n + \frac 32}$.  
In that case, we use Corollary  \ref{cor1}
or Corollary \ref{cor1.2d} to tell us that
$\Es_{\eta'}(z) \asymp \Es_{\tilde \eta}(z)$ and hence the
quantity in \eqref{jun12.6} is bounded below by a constant
times
\[ 
F_{\eta_n} \,F_{\tilde \eta} (2d)^{-|\eta_n|-|\tilde \eta|}\,
\Es_{\tilde \eta}(z)
\sum_{\eta^*\subset C_{n + \frac 32}} (2d)^{-|\eta^*|} \,  F_{ \eta^* } (A) 
.
\] 
We also use the results of that section to tell us that
\[  \sum_{\eta^*} (2d)^{-|\eta^*|} \,  F_{ \eta^* 
} (A) 
\geq c \, H_{\p A}(y,w).\]
Therefore, using Corollary \ref{cor3}, we see that for $d \geq 3$,
\begin{eqnarray*}
\sum_{\eta^*} \mu_m(\eta \oplus \eta^* \oplus
\tilde  \eta )  & \asymp &  F_{\eta} \,F_{\tilde \eta} (2d)^{-|\eta|-|\tilde \eta|}
\, 
H_{\p A}(y,w)\\
& \asymp & F_{\eta} \,\Es_\eta(y) \,(2d)^{-|\eta|}
\, F_{\tilde \eta} (2d)^{- |\tilde \eta|}
\,  H_{\Z^d \setminus \tilde \eta}(0,w),
\end{eqnarray*}
and similarly for $d=2$ with $F_{\tilde \eta}$ replaced
with $F_{\tilde \eta}^{*,n}$.  If $d \geq 3$, then
$ H_{\Z^d \setminus \tilde \eta}(0,w) \asymp
H_{\p(\hat \Z^d \setminus \tilde \eta)}(0,w).$

We now claim that for $d=2$, $F_\eta \asymp n \, F_{\tilde \eta}^{*,n}$, in
other words, the measure of loops that intersect $\eta$ and also disconnect
$0$ from $\p C_n$ equals $\log n + O(1)$.  Indeed, this follows from
Lemma \ref{sep14.lemma1}.  We therefore get
\[   F_{\tilde \eta}^{*,n} \,  H_{\Z^d \setminus \tilde \eta}(0,w)
\asymp n^{-1} \, 
 F_{\tilde \eta}  \,  H_{\Z^d \setminus \tilde \eta}(0,w)
  \asymp F_{\tilde \eta} \, H_{\p(\hat \Z^d \setminus \tilde \eta)}(0,w).\]
Hence, for all $d \geq 2$,
\[  \sum_{\eta^*} \mu_m(\eta \oplus \eta^* \oplus
\tilde  \eta )  \asymp 
F_{\eta} \,\Es_\eta(y) \,(2d)^{-|\eta|}
\, F_{\tilde \eta} (2d)^{- |\tilde \eta|}
\,  H_{\p(\hat \Z^d \setminus \tilde \eta)}(0,w)
 \asymp \mu_n(\eta) \, \mu_{n,m}(\tilde \eta).\]

\end{proof}

It is useful to view the measures $\mu_n$
as generating a Markov chain $\gamma_n$  with state space
\[                  \saws := \bigcup_{n=0}^\infty  \saws_n. \]
The transitions always go from $\saws_{n}$ to $\saws_{n+1}$,
and are such that $\gamma_n \prec \gamma_{n+1}$.
Using  \eqref{jun3.1} we give the transitions by 
\begin{equation}  \label{mutransition}
\phi(\gamma_n,\gamma_{n+1}) := \frac{ \mu_n(\gamma_{n+1})}
{\mu_n(\gamma_{n})} = (2d)^{-|\tilde \gamma|} F_{\tilde\gamma}
(\Z^d \setminus \gamma_{n}) \, \frac{\Es_{\gamma_{n+1}}(z_{n+1})}{
\Es_{\gamma_n}(z_n)}.
\end{equation}
Here $z_n,z_{n+1}$ are the terminal points of $\gamma_n,\gamma_{n+1}$,
respectively,
and we have written $\gamma_{n+1} = \gamma_n \oplus \tilde \gamma$.

%
%

\subsection{Coupling a one-sided LERW}  \label{onecouplesec}

Before handling the case of pairs of walks, it is useful to consider the simpler
question of coupling one-sided infinite LERW with different initial conditions.
The one-sided LERW is a probability measure on $\bar \saws_0$.  For each $n$, we
write $\eta \in \bar \saws_0$ uniquely as $\eta = \eta_n \oplus \eta_n^*$ where
$\eta_n \in \saws_n$.  We will write $\eta \sim_k \tilde \eta$ if in this decomposition
 $\eta_k^* = \tilde \eta_k^*$, that is, if the paths agree after their first
visit to $\p C_k$.  We do not require that $\eta_k$ and $\tilde \eta_k$ have
the same number of steps.  If $\eta_n , \tilde \eta_n \in \saws_n$,
we write $\eta_n = _j \tilde \eta_n$ if the paths agree from the first visit
to $\partial C_{n-j}$ onward. 

\begin{proposition}  There exist  $0 < u , c< \infty $ such that if $\eta_n,\tilde \eta_n
\in \saws_n$, then we can couple $\eta, \tilde \eta$ on the same probability
space such that
\begin{itemize}
\item  The distribution of $\eta$ is  LERW  conditioned to start with $\eta_n$.
\item   The distribution of $\tilde \eta$ is LERW conditioned to start with
$\tilde \eta_n$. 
\item  If  $J$ denotes the smallest integer $k$ such that 
$\eta^* \sim_{n+k} \tilde \eta^*$, then  $\E[e^{ u J}]  \leq c $. 
 \end{itemize}
 \end{proposition}
 
We start with a preliminary lemma.

\begin{lemma} \label{feb6.lemma1}
 There exist $c'  < \infty$ such that if $\eta_n, 
\tilde \eta_n \in \saws_n$ with $\eta_n = _k \tilde \eta_n $ with
 then we can couple 
 $\eta, \tilde \eta$ on the same probability
space such that
\begin{itemize}
\item  The distribution of $\eta$ is  LERW  conditioned to start with $\eta_n$.
\item   The distribution of $\tilde \eta$ is LERW conditioned to start with
$\tilde \eta_n$.
\item  
$    \Prob\{\eta_n^* = \tilde \eta_n^*   \}  \geq 1 -
      c' \, e^{-k}.$
 \end{itemize}
 Moreover, if $k \geq 1$, then  $\Prob\{\eta_n^* = \tilde \eta_n^*   \} \geq 1/c$.
 \end{lemma}
  
\begin{proof}  We assume $k \geq 1$, 
The distribution of $\eta_n^*, \tilde \eta_n^*$ given $\eta_n, \tilde \eta_n $,
is that of the loop erasure of a random walk starting at the endpoint 
conditioned to avoid  $\eta_n , \tilde \eta_n $, respectively.  Lemmas \ref{cor1} and \ref{cor1.2d}
show that we can couple these conditioned
random walks so that they agree up to an event
of probability $O(e^{-k})$. 

\end{proof}

\begin{lemma} \label{feb7.lemma2}
 There exists $ \delta > 0$ such that if
$\eta_n, \tilde \eta_n \in \saws_n$, then we can define
$\eta,\tilde \eta$ on the same probability space so that
\begin{itemize}
\item  The distribution of $\eta$ is  LERW  conditioned to start with $\eta_n$.
\item   The distribution of $\tilde \eta$ is LERW conditioned to start with
$\tilde \eta_n$.
\item With  probability
at least $\delta$, 
 \[  \eta_{n+2}^* = \tilde \eta_{n+2}^*, \;\;\;\;
  \eta_{n+2}^* \setminus \eta_{n+j}^* \subset
    \{\x
      \in\Z^d \setminus C_{n+1} : x_1 \geq |\x|/10\}  
.\] 
\end{itemize}
\end{lemma}

\begin{proof}
We let $\omega, \tilde \omega$ denote
simple random walks conditioned to avoid $\eta_n, \tilde \eta_n$
respectively.  We first let $\omega, \tilde \omega$ move independently
until they reach $\partial C_{n+1}$.  For each one there is a positive
probability that the walk did not enter $C_{n-1}$ and that the endpoint
is within distance $e^{n+1}/20$ of $\e_{n+1}$.  The distribution of
the endpoint is comparable to harmonic measure, that is, comparable
to $e^{-n (1-d)}$ for each point. 

Given that $\omega,\tilde \omega$ have reached $\partial C_{n+1}$,
there is a universal $\rho > 0$ such that the probability (conditioned that
it avoids $\eta_n $ or $\tilde \eta_n$) that the rest
of the path avoids $C_n$ is greater than $\rho$.  We consider the set
of paths with this property, and we can now couple $\omega, \tilde \omega$
with positive probability such that on this event, the distribution of the
remainder of the path is random walk conditioned to avoid $C_{n}$. 
We will write $\omega^*$ for the future in the coupled walks.  This
is a walk starting on $\p C_{n+1}$ within distance
 $e^{n+1}/20$ of $\e_{n+1}$.  We write $\sigma_r$ for the first visit of
 $\omega^*$ to $\partial C_{n +r}.$

Consider the event that all the following
 holds.
 \begin{itemize}
 \item  the walk reaches
 $\partial C_{n+j +1}$ without leaving $ \{\x
      \in\Z^d \setminus C_{n+1} : x_1 \geq |\x|/10\}  \}$
      \item  After this time it never returns
      to $C_{n+j}$.
      \item  there is 
     a cut time for $\omega^*[\sigma_0, \sigma_2]$
     that occurs between time  $\sigma_{4/3}$ and $\sigma_{5/3}$.
     \item $\omega^*[\sigma_{4/3},\infty)$ never visits
      $\partial C_{n + 1}$
        \item $\omega^*[\sigma_2,\infty)$ never visits
          $\partial C_{n +  \frac 53}$

     \end{itemize}
     Under this event, the cut time is also a cut time for
     the entire path (with either $\eta_n$ or $\tilde \eta_n$
     as
  initial condition).  Hence the loop erasure is the same after
  that point, and, in particular, the loop erasure after the the
  first visit of the loop erasure to $\partial C_{n+2}$ is the same.
 All we need is that the probability of this event is bigger than some
 $\epsilon_j > 0$ and this is easy to verify.  (We could get a lower
 bound for the probability in terms of $j$ but we will not need it.)

\end{proof}  

We can now describe the coupling.  Let $r$ be sufficiently large
so that $c'\sum_{k \geq r-2} \leq 1/2$.  Let 
 \[   q_i = \sum_{k=ir}^{(i+1)r - 1} c' \, e^{-k}, \]
 and note that 
 \[   q_i \leq c'' e^{-ir} , \;\;\;\;\;  \sum_{i=1}^\infty q_i \leq \frac 12. \] We start with $(\eta_n,\tilde \eta_n)$
and we will recursively define $\gamma_m = \eta_{n + mr}, \tilde \gamma_m = \tilde \eta_{n + mr} $ 
 for $m=0,1,\ldots$
and a nonnegative integer valued random variable $K_m$.
  At each
stage we will have
%
%
%
 \begin{itemize}
\item  $ \gamma_{m-1}  \prec \gamma_m, \;\;\;\;\;  \tilde \gamma_{m-1} \prec 
 \tilde \gamma_m ,$
\item $\gamma_m$ ($\tilde \gamma_m$)
 has the distribution of a LERW conditioned to start with $\eta_n$ (resp., $\tilde 
 \eta_n$)
stopped at its first visit to $\partial  C_{n + mr}$.
\item  If $K_m = k$, then $\gamma_m =_{kr -2}  \tilde \gamma_m.$
\end{itemize} 
Using the lemmas above, we can couple so that the following is true
given $\gamma_m,\tilde \gamma_m$. 
\begin{itemize}
\item  If $K_m = k$, we can define $\gamma_{m+1}, \tilde \gamma_{m+1}$
such that, except for an event of probability at most
$q_k$, 
  $\gamma_{m+1} =_{r(k+1) - 2} \tilde \gamma_{m+1}$.  If the
last equality holds, 
  we set $K_{m+1} = k + 1$.  Otherwise, we set $K_m = 0$. 
\item  If $K_m = 0$, then  we can
define $\gamma_m, \tilde \gamma_m $ so that with probability at least
$\delta$, $\gamma_m =_{r-2} \tilde \gamma_m$.  On the event that this happens,
we set $K_{m+1} = 1$.  Otherwise, we set $K_{m+1} = 0$. 
\end{itemize}
Let $T = \sup\{m: K_m = 0\}$.  Note that $J \leq rT$,
  Then these assumptions imply (see following lemma)
that there exists $u = r \beta> 0$ with $\E[e^{uJ}]
   \leq  \E[e^{\beta T}] < \infty$.

\begin{lemma}
Suppose $X_0 = 0,X_1,X_2,\ldots$ is a sequence of nonnegative integer random
variables adapted to a filtration $\{\F_n\}$
 such that for each $n$, $X_{n+1} = X_n + 1$ or $X_{n+1} = 0$. 
Suppose there exists $0 < \delta , c, \alpha < \infty$ such that for all $n,j$
\[      \Prob\{X_{m} > 0 \mbox{ for all  } m > n   \mid \F_n\} \geq \delta, \]
\[       \Prob\{X_{n+1} = 0 \mid \F_n\} \leq c \, e^{-\alpha X_n}. \]
Let $T = \max\{n: X_n = 0\}$.  Then
there exists $\beta = \beta(\delta,c,\alpha) $ such that
$\E[e^{\beta T}] < \infty.$  
\end{lemma}

\begin{proof}
Let $\sigma_0 = 0$
and $\sigma_k = \min\{n < \sigma_{k-1}: X_n = 0 \}$.  Then $\Prob\{\sigma_1 <  \infty\}
\leq 1- \delta$ and by iterating $ \Prob\{\sigma_k <  \infty\}
\leq (1- \delta)^k$.  Hence $\Prob\{T < \infty\} = 1$ and we can write
\[   \E[e^{\beta T}] =  \sum_{k=0}^\infty \E\left[e^{\beta \sigma_k};
\sigma_k < \infty, \sigma_{k+1} = \infty \right] \leq \sum_{k=0}^\infty  \E\left[e^{\beta \sigma_k};
 \sigma_k < \infty  \right].\]  Since
 $\Prob\{\sigma_1 = n\}  \leq \Prob\{\sigma_1 = n \mid \sigma_1
 > n-1\} \leq c \,e^{-\alpha (n-1)},$
  for $\beta$ sufficiently small, 
\[    \E\left[e^{\beta \, \sigma_1}; \sigma_1 < \infty \right]  \leq  1- 
  \frac \delta 2.\]
By iterating this, we see for $k \geq 1$,  
$ \E\left[e^{\beta \sigma_k} ; \sigma_k < \infty\right] \leq (1-\frac \delta 2)^k . $
\end{proof}

\subsection{Pairs of walks}  \label{pairsec}

If $n  < m-1$, and $\eta  \in \saws_m$, we  define $\eta_n, \eta^*,
\eta_{n+1,m}$
by the decomposition \eqref{decomp2},
\[     \eta  = \eta_n \oplus \eta^*  \oplus   \eta_{n+1,m}, \]
where $\eta_n \in \saws_n,  \eta_{n+1,m} \in \saws_{n+1,m}$ and
$\eta^* $ is the middle.  
If $\bfeta = (\eta^1,\eta^2)
\in \saws_m^2$, we similarly
write 
\[      \bfeta  = \bfeta_n \oplus \bfeta^*  \oplus \bfeta_{n+1,m} \]
where the decomposition is done separately on $\eta^1,\eta^2$.  
Proposition \ref{indprop} implies that if $m \geq n+2$,
\begin{equation}  \label{compmeasure}
\sum_{\bfeta = \bfeta_n
\oplus \bfeta^*\oplus \bfeta_{n+1,m}} \mu_m(\bfeta)
\asymp 
\mu_n(\bfeta_n) \, \mu_{n+1,m}(\bfeta_{n+1,m}), 
\end{equation}

\begin{itemize}

\item Recall that $Q_n  $ is defined
on $\saws_n^2 $ by
\[   Q_n(\bfeta) = 1\{\bfeta \in \pairs_n\} \,
        e^{-L_n(\bfeta)} = e^{-L_n(\bfeta)},\] 
       where $L_n(\bfeta)$ is the loop measure of loops in $\hat C_n$
that intersect both $\eta^1$ and $\eta^2$.  If
$d=2$, we only consider loops that do not disconnect
$0$ from $\p C_n$.  By definition, $L_n(\bfeta)
 = -\infty$ if $\eta^1 \cap \eta^2 \neq \{0\}$.
If $n \leq m-1$, we also view $Q_n$ as defined on $\saws_m^2$ by
\[    Q_n(\bfeta_n
\oplus \bfeta^*\oplus \bfeta_{n+1,m} ) = Q_n(\bfeta_n).\]


\item We define   $\bar Q_{n+1,m}$ on $\saws_m^2$    by
\[    \bar Q_{n+1,m}(\bfeta_n
 \oplus \bfeta^*\oplus \bfeta_{n+1,m} ) =
        e^{-L_{m}(\bfeta_{n+1,m}  )}, \]
where
\begin{itemize}
\item $ L_{m}(\bfeta_{n+1,m} ) = -\infty$
if $\eta_{n+1,m}^1 \cap \eta_{n+1,m} ^2 
\neq \eset$,
 \end{itemize}
 and
\begin{itemize}
\item   ($d\geq 3$) 
$L_{m}(\bfeta_{n+1,m})$ is the loop measure of loops in $\hat C_m$ that  
intersect both $\eta^1_{n+1,m}$ and $\eta^2_{n+1,m}$,
\item($d = 2$)  
$L_{m}(\bfeta_{n+1,m})$ is the loop measure of loops in $\hat C_m$ that  
intersect both $\eta^1_{n+1,m}$ and $\eta^2_{n+1,m}$ and do not disconnect $0$ from $\p C_m$. 
\end{itemize}



\item If $\bfeta_n \prec \bfeta$, we define
\[         \lambda_m(\bfeta\mid \bfeta_n) = \frac{\lambda_m(\bfeta)}
{\lambda_n(\bfeta_n)} =  \frac{Q_m(\bfeta)\, \mu_m(\bfeta)}
   {Q_n(\bfeta_n) \, \mu_n(\bfeta_n)} \leq
   \frac{ \mu_m(\bfeta)}
   { \mu_n(\bfeta_n)}=
    \mu_m
   (\bfeta\mid \bfeta_n), \]
so that
\[    \lambda_m(\bfeta) = \lambda_n(\bfeta_n) \, \lambda_m(\bfeta
 \mid \bfeta_n). \]
If $ n < m$ and $\bfeta \in \pairs_n$,  let
\[  \lambda_m(\bfeta)
= \sum_{\bfeta \prec \bfeta' \in \pairs_m}
   \lambda_m(\bfeta') = \lambda_n(\bfeta)
    \sum_{\bfeta \prec \bfeta' \in \pairs_m}
   \lambda_m(\bfeta' \mid \bfeta) .\]

 \item  Let $\seppair_{n,m}$  be the set of $\bfeta = (\eta^1,\eta^2)
\in \pairs_m $ with
\[      \bfeta^* \subset C_{n+2} \setminus C_{n-1},  \]
\[     \dist\left[\eta^{j,*} , \eta^{3-j}\right]  \geq e^{n-3}, \;\;\;\;
j=1,2.\]

\end{itemize}  

\begin{proposition} There exist $0 < c_1 < c_2 < \infty$ such
that if  $n < m-1$ and $\bfeta \in \pairs_n$,
\begin{equation}  \label{qsub}    c_1 \, 1\{\bfeta \in \seppair_{n,m}\} \, {Q_n(\bfeta)}\,
\bar Q_{n+1,m}(\bfeta)\, \leq
       {Q_m(\bfeta)}
               \leq c_2\,  {Q_n(\bfeta)}\, \bar Q_{n+1,m}(\bfeta) . 
               \end{equation}
 \end{proposition}

\begin{proof}  For the upper bound we note that
\[ \frac{Q_m(\bfeta)}
   {Q_n(\bfeta)\, \bar Q_{n+1,m}(\bfeta)} \leq \exp\{L'\} , \]
   where $L'$ is the 
  measure of loops $\ell$ that intersect both
  $\eta^1_n$ and $\eta^2_n$ and also intersect both  
 $\eta^1_{n+1,m}$ and $\eta^2_{n+1,m}$.  For $d=2$, it is
 also required that the loops not disconnect $0$ from $\p C_m$.
 In particular, such loops must intersect both $C_n$ and $\p C_{n+1}$.
 If $d \geq 3$, Lemma \ref{june6.lemma1} tells us that the
 measure of such loops is uniformly bounded.
Similarly, for  $d=2$, Lemma \ref{sep14.lemma1} tells us that
the measure of such nondisconnecting loops is bounded. 

For the lower bound, note that if $\bfeta \in \seppair_{n,m}$, then 
\[ \frac{Q_m(\bfeta)}
   {Q_n(\bfeta)\, \bar Q_{n+1,m}(\bfeta)} \geq \exp\{-L''\} , \]
   where $L''$ is the 
  measure of loops $\ell$  that intersect $C_{n+2} \setminus C_{n-1}$
  and are of diameter at least $e^{n-3}$; for $d=2$, we also require
  the loops to be nondisconnecting.  Again, Lemmas
  \ref{june6.lemma1} and \ref{sep14.lemma1} give uniform upper
  bounds for $L''$.
 \end{proof}

%
%
%

One of the most important tools in understanding $\lambda_n$
is the separation lemma.  This says the (almost obvious) fact
that if two paths are conditioned to avoid each other then
their endpoints tend to be far apart.  There are many versions
that can be used.  We will define a particular separation event.
The choice of $1/10$ is arbitrary but it is convenient to choose
a fixed small number.

\begin{definition}
$\;$
\begin{itemize}

\item If $\eta \in \saws_n$,
let $I_n(\eta)$ be the indicator function
of the event  
\[    
    \eta \cap (C_{n} \setminus C_{n-(1/10)})
 \subset \{x = (x_1,x_2,x_3): x_1  \geq e^{ -1}\, |x|\}.\]

\item Let $\Sep_n$ be the set of $\bfeta = (\eta^1,\eta^2) \in \pairs_n$
such that  $I_n(\eta^1)  = 1$ and
\, $I_n(-\eta^2) = 1$.

\item  If $n < m$ and $\eta' \in   \saws_{n, m}$, let  $I _{n,m}(\eta')$
be the indicator function
of the event  
\[    
    \eta' \cap (C_{n+(1/10)} \setminus C_{n })
 \subset \{(x_1,x_2,x_3): x_1  \geq e^{n-1}\}.\]

 \item   If $n < m$, let $\Sep_{n,m}$ be the set
 of $\bfeta = (\eta^1,\eta^2) \in  \bar \saws_{n,m}^2$ such
that   $I_{n,m}(\eta^1) =1$  and $ I_{n,m}(-\eta^2) = -1$.

\end{itemize}

\end{definition}

%
%

We will have two separation lemmas.  The first is stronger and
deals with the endpoint of the beginning of the path.  The second
is not as strong (we could prove the stronger result but do not
need it) and deals with the initial part of the final piece
of the path.  Various version of the separation lemma can be found
in \cite{Masson} and \cite{LV} in the two-dimensional case and \cite{Shir}
in the three-dimensional case.  For completeness we include a proof
of one version in the appendix. 

\begin{lemma}[Separation Lemma I]  There exists $c > 0$ such that if
$ 2 \leq n \leq m - 1$, $\bfeta
\in \saws_n$, and
\[     \lambda_m^\Sep(\bfeta) =
\lambda_n(\bfeta)  \sum_{\bfeta' \in \Sep_m, \bfeta \prec \bfeta'}
      \lambda_m(\bfeta' \mid \bfeta). \]
Then $  \lambda_m^\Sep(\bfeta) \geq c \,   \lambda_m (\bfeta).$
\end{lemma}

\begin{lemma}[Separation Lemma II]  There exists $c > 0$ such that
if $n < m$, then
\[      \sum_{\bfeta \in \pairs_m, \bfeta \in \Sep_{n,m}}
       \mu(\bfeta) \, \bar Q_{n,m}(\bfeta')
       \geq c \, \sum_{\bfeta \in \pairs_m}
       \mu(\bfeta) \, \bar Q_{n,m}(\bfeta').\]
\end{lemma}

We will now consider some easy consequences.

\begin{proposition}  \label{sep26.1}
There exist constants $0 < c_1 < c_2 < \infty$ such that the
following holds.
\begin{enumerate}

\item  If $\bfeta = (\eta^1,\eta^2)  \in \Sep_n$, then $ \lambda_{n+1}(\bfeta) \geq
c_1\, \lambda_n(\bfeta)$,

\item  For every $n \geq 0$, \[ c_1 \,\lambda_n \leq \lambda_{n+1} \leq \lambda_n.\]
\item   If $n  < m-1$, 
\[        \lambda_{n+1,m} \asymp \frac{ \lambda_m}{ \lambda_n}. \]
and, more generally, $ \lambda_{m}(\bfeta) \geq c_1 \, \lambda_m/\lambda_n$
for $\bfeta \in   \Sep_n$.

\item If $\bfeta \in \pairs_n$, $ \lambda_{m}(\bfeta) \leq c_2  
\,  \lambda_n(\bfeta)\, (\lambda_m/\lambda_n). $

%
%
\end{enumerate}

\end{proposition}

\begin{proof}$\;$

\begin{enumerate}

\item  We use Corollary \ref{extendcor} to see that with positive
$\mu_{n+1}$ probability the extension of $\bfeta$ will still
be separated. 
\item This follows from part 1 and the Separation Lemma I.

\item  Here we use both Separation Lemma I and Separation Lemma II.

\item  This is done similarly.



\end{enumerate}

\end{proof}

\begad

Separation lemmas are key tools for many problems.  They
can be considered generalizations of ``boundary Harnack principles''.
We will not discuss this, but just say that the idea is that if you have
a bounded domain, start a process very near the boundary, and condition
the process to not leave the domain in, say, one unit of time, then
the process will get away from the boundary very quickly.  Although the probability
of escaping the boundary is small, the probability of staying near the
boundary without exiting is of a smaller order of magnitude.  

The analogue for us of being near the boundary is to say that pair
of walks are close to each other near their terminal points.  Once
the paths ``separate'' somewhat, then there is a reasonable chance
that they will stay separated.  

\endad

\subsection{Coupling the pairs of walks}  \label{couplesec}

In this section we fix a (large) integer $N$.  Our goal is to
couple the probability measures $\lambda_N^\#(\bfeta)$ and $\lambda_N^\#(
\bfeta')$ for different starting configurations $\bfeta,\bfeta'$.  We use
a coupling strategy similar to that in Section \ref{onecouplesec} although here
we will only go up to level $N$ rather than to infinity.
We start by giving some 
notation; in order to make it easier to read, we  will leave $N$ implicit.
\begin{itemize}

\item  Let $  b_n = \lambda  _{n,N}.$
If $n \leq N$ and $\bfeta \in \pairs_n$, let
\[  b(\bfeta) =  \frac{\lambda_N(\bfeta)}{\lambda_n(\bfeta)}
= \sum_{\bfeta' \in 
\pairs_N, \bfeta \prec \bfeta'} \lambda_N(\bfeta' \mid
\bfeta) .\]
If $\bfeta \in \pairs_N$, then $b(\bfeta) = 1$.
Note that if $n < N-1$, 
\begin{equation}  \label{densitycheck}
b(\bfeta) =   \sum_{\bfeta' \in \pairs_{n+1},
\bfeta \prec \bfeta'} \lambda(\bfeta' \mid
\bfeta)  \, b(\bfeta').
\end{equation}

\item  If $\bfeta, \tilde \bfeta \in \pairs_n$ we write
$\bfeta=_j  \tilde \bfeta$ if the paths agree from the first visit
to $\partial C_{n-j}$ onwards.  In other words, if we write
\[         \bfeta = \bfeta_{n-j} \oplus \bfeta' ,
\;\;\;\;
\tilde  \bfeta = \tilde \bfeta_{n-j} \oplus \tilde \bfeta'  ,\]
with $\bfeta_{n-j}, \tilde \bfeta_{n-j} \in \pairs_{n-j}$,
then $ \bfeta'  = \tilde \bfeta'.$

\end{itemize}

We will consider  Markov
chains $\bgamma_n, n=k,k+1,\ldots,N$ 
taking values in
$    \pairs :=\bigcup_{i=1}^\infty \pairs_i $
with the properties:

\begin{itemize}

\item For all $n$,  $\bgamma_n \in \pairs_n$.  If $j < n$, then  
$\bgamma_j \prec \bgamma_n$.

\item  The transitions are given by 
\begin{equation}  \label{chaindensity}
    p(\bgamma_n, \bgamma_{n+1})
=   \lambda( \bgamma_{n+1} \mid \bgamma_n  )
\, \frac{   b(\bgamma_{n+1} )}  {   b(\bgamma_n)} 
\leq  \mu( \bgamma_{n+1} \mid \bgamma_n  )
\, \frac{  b(\bgamma_{n+1} )}  {   b(\bgamma_n)} 
\end{equation}

\end{itemize}

This formula assumes that
$  b(\bgamma_n) > 0$, but we start with this condition
at time $k$ and hence with probability one this will
hold for all $n$. 
Note that \eqref{densitycheck} shows that this is a  
well-defined transition probability. We recall that  $   b(\bgamma_{n+1} )
\leq c_2 \,   b_{n+1}$ and if $\bgamma_{n+1}
\in \Sep_{n+1}$, then  $  b(\bgamma_{n+1} )
\geq c_1\,   b_{n+1}$

We now describe the
coupling which will be a construction of
ordered pairs $\bgamma^*_n = (
\bgamma_n, \tilde \bgamma_n)$ such that $\bgamma_n,
\tilde \bgamma_n$ both follow the Markov chain although
with different initial
distributions on $\bgamma_n,\tilde \bgamma_n$. More precisely,
we   define a   coupling
$X_n = (\bgamma_n,\tilde\bgamma_n,J_n)$ to be
random variables defined on the same probability space
such that the following hold.
\begin{itemize}
\item  $(\bgamma_k,\bgamma_{k+1},\ldots,\bgamma_N)$ is
a Markov chain satisfying the transition probabilities
given by \eqref{chaindensity}.
\item  Similarly, 
$(\tilde \bgamma_k,\tilde \bgamma_{k+1},\ldots,\tilde \bgamma_N)$ is
a Markov chain satisfying \eqref{chaindensity} although
the initial distribution may be different.
\item  $J_n$ is a nonnegative integer random variable
with the property that if $J_n = j$, then
\[     \bgamma_n = _j \tilde \bgamma_n, \;\;\;\;
  b(\bgamma_n) \geq e^{-j/4} \,   b_n.\]
\item  $J_k = 0$, and for every $k \leq n <N$, either $J_{n+1} = J_n + 1$
or $J_{n+1} = 0$.

\end{itemize}

We can now state the main result.

\begin{proposition} \label{coupling}
There exists $\alpha > 0, c < \infty$
such that for all $k,n$ and all $N \geq 2n + k$, then
for any initial distributions on $\bgamma_k,\tilde \bgamma_k$
we can find a coupling such that
\[     \Prob\{J_N \leq n\} \leq c \, e^{-\alpha n}.\]
In particular, except for an event of probability
$O(e^{-\alpha n})$,
$                \bgamma_N =_n \tilde \bgamma_N.$
\end{proposition}

We collect some of the lemmas from previous sections here.
These are the facts that we will need in this subsection.
We have already done most of the work in establishing these results
so we will be brief in our proof.

\begin{lemma}  There exist $0 < c_1 < c_2 < \infty$ and
$\beta > 0$ such that the following holds.
\begin{enumerate}

\item  For any $\bgamma_n$, 
\[ \sum_{  \bfeta
\cap C_{n - (j/2)} \neq \eset}
\lambda(\bgamma_n \oplus \bfeta \mid \bgamma_n)
\leq c \, e^{-j/2}.\]

\item  If $\bgamma_n =_j \tilde \bgamma_n$, and $\bfeta
\cap C_{n - (j/2)} = \eset$, then
\[        \lambda(\bgamma_n \oplus \bfeta
\mid \bgamma_n) = \lambda(\tilde \bgamma_n
\oplus \bfeta \mid \tilde \bgamma_n) 
 \, [1 + O(e^{-j/2})], \] 
%
%
\item  If $\bgamma_n =_j \tilde \bgamma_n$, then
\[   | b(\bgamma_n) -   b(\tilde \bgamma_n) |\leq c_2
\, e^{-j/2} \,   b_n.\]
In particular, if $  b(\bgamma_n) \geq e^{-j/4} \, 
b_n$, then
\[    b(\bgamma_n) =   b(\tilde \bgamma_n) \, [1
+ O(e^{-j/4})],\]
\[   \sum_{  \bfeta
\cap C_{n - (j/2)} \neq \eset} p (\bgamma_n,
\bgamma_n \oplus \bfeta) \leq c \, e^{-j/4}, \]
and if also $ b(\bgamma_n \oplus \bfeta ) \geq e^{-(j+1)/4} \, 
b_{n+1}$, then
\[ p(
\bgamma_n, \bgamma_n \oplus \bfeta) = p(
\tilde \bgamma_n, \tilde \bgamma
\oplus \bfeta) 
 \, [1 + O(e^{-j/4})].\]

\item  If $n \leq N-1$, then given $\bgamma_n$,
the probability that $ \bgamma_{n+1}  \in \Sep_{n+1}$     is
at least $c_1$.  In other words, for every $\bgamma_n$,
\begin{equation}  \label{canadaday.1}
 \sum_{\bfeta \in \Sep_{n+1}}
       p(\bgamma_n,\bgamma_n \oplus \bfeta) \geq c_1
      .
      \end{equation}
In particular,
\[    \sum_{\bfeta \in \Sep_{n+1}}
       p(\bgamma_n,\bgamma_n \oplus\bfeta) \, b(\bfeta)
         \geq c \,  b_{n+1}.\]
\[ \sum_{\bfeta, b_{n+1}(\bfeta \oplus \bfeta)
\leq e^{-(j+1)/4} } p(\bgamma_n,\bgamma_n \oplus\bfeta)
  \leq c \, e^{-j/4}.\]
\end{enumerate}

\end{lemma}

\begin{proof}$\;$

\begin{enumerate}

\item   Write $\bfeta = \bfeta^* \oplus \tilde \bfeta'$ where
$\tilde \bfeta' \in \pairs_{n+1,N}.$  and use
\[    \lambda(\bgamma_n \oplus \bfeta' \oplus \tilde \bfeta \mid \bgamma_n)
\leq \mu(\bgamma_n \oplus \bfeta' \oplus \tilde \bfeta \mid \bgamma_n)
\, Q_{n+1,N}(\tilde \bfeta).\]
Using Corollary \ref{cor2}  we can see that given $\bgamma$ and $\tilde \bfeta$,
the $\mu$-probability that $\bfeta^* \cap C_{n-(j/2)} \neq \eset$
is $O(e^{-j/2})$.  We then use \eqref{qsub}.

\item  This follows from a
combination of Lemma \ref{june6.lemma1} and Corollary \ref{cor1}.

\item We write
\[  b(\bgamma_n) = \sum_{   
\bgamma_n \oplus  \bfeta\in \pairs_N} \lambda_N( \bgamma_n \oplus \bfeta \mid \bgamma), \]
and similarly for $b(\bgamma_n')$ and use the first two parts.

\item This follows from the separation lemma.

\end{enumerate}           

%
%
  \end{proof}


\begin{lemma}  For every $j > 0$, there exists $\delta_j >0$
such that the following holds.  Suppose $k \leq N - (j+2)$
and initial conditions $\bgamma_k,\tilde \bgamma_k$ are
given.
Then we can couple $\bgamma^* = (\bgamma,\tilde \bgamma)$
on the same probability space such that
with probability at least $\delta_j$, we have
$I(\bgamma_{k+j+2}) = 1$ and $\bgamma_{k+j+2}
=_j \tilde \bgamma_{k+j+2}.$

\end{lemma}

Note that  $I(\bgamma_{k+j+2}) = 1,\bgamma_{k+j+2}
=_j \tilde \bgamma_{k+j+2}$ imply that 
$I(\tilde \bgamma_{k+j+2}) = 1$.   It will not be important
to give estimates for $\delta_j$ in terms of $j$; the coupling
works as long as $\delta_j > 0$ although the exponent $\alpha$
does depend on the actual values.  The proof of this is
similar to Lemma \ref{feb7.lemma2}  although we first
separate the paths.   We do this so that
  the measure of loops that
intersect the extensions of the paths will be bounded.  

\begin{proof}  We fix $j$.
In this proof all constants (implicit or explicit) or
phrases like ``with positive probability'' mean that there
exist constants that can be chosen uniformly over
all $k$ and all $\bgamma_k, \tilde \bgamma_k$ (although
they may depend on $j$).

By \eqref{canadaday.1},
we can see that there is a positive probability that
$\bgamma_{k+1}, \tilde \bgamma_{k+1} \in \Sep_{k+1}$.
Let 
\[   \cone^1 = \{x = (x_1,x_2,x_3): x_1 \geq e^{-1} \, |x| \},
\;\;\;\; \cone^2  =
\{x = (x_1,x_2,x_3): x_1 \leq -e^{-1} \, |x| \} .\]

Recall that $\mu(\bgamma_j \mid \bgamma_{k+1})$
is obtained by taking  simple random walks starting at the endpoint
$z_{k+1}^i$
of each $\gamma_{k+1}^i$ conditioned to avoid $\gamma_{k+1}^i$
and then erasing loops.  Let us consider the first
time that these random walks reach $\p C_{k+(6/5)}$.  Using
Lemma \ref{importantlemma}
we can see that with positive probability the random walk
avoiding 
$\gamma^1_{k+1}$ stays in 
\[    \cone_*^1 :=\{x = (x_1,x_2,x_3): x_1 \geq e^{-2} \, |x| \},\]
and its terminal point is in $\cone^1$.  The same is
true for the  walk avoiding $\gamma^2_{k+1}$ staying
in $\cone_*^2 =\{x = (x_1,x_2,x_3): x_1 \leq -e^{-2} \, |x| \}$
with terminal point in $\cone^2$.   Using the Harnack principle,
we can see that the distribution of the terminal point in $\cone^i$
is comparable
to the harmonic measure from $0$ of $\p C_{k + (6/5)}$ which in
turn is comparable to $e^{ k(1-d)}. $

For a simple random walk $S$ starting on $\cone^i   \cap \p C_{k + (6/5)}$,
conditioned to avoid $\gamma^i$, 
there is a positive probability that the following holds:
\begin{itemize}
\item  The walk never visits $C_{k+1}$.
\item  The intersection of the walk with $C_{k + j+ 2}$ is contained
in $\cone^i_*$.
\item  The intersection of the walk with $\p C_{k+j+2}$ is contained
in $\cone^i$.
\item  There is a time $t$ with $S_t \in C_{k+(8/5)} \setminus
C_{k+(7/5)}$ that is a cut point for the random walk. (See
\cite{rwcutpoint} for existence of cut points for random walks
in $\Z^2$ and $\Z^3$).
\item  The walk never returns to $C_{k + (9/5)}$ after reaching
$\p C_{k+2}$ for the first time. 
\end{itemize}

Using this we can see that we can couple the conditional
simple random walks avoiding $(\gamma^1_{k+1},
\gamma^2_{k+1})$ with those avoiding
$(\tilde \gamma^1_{k+1},\tilde \gamma^2_{k+1})$ such that 
with positive probability the
paths agree from their first visit to  $\p C_{k + (6/5)}$ onward
and they lie in the event described above.  In particular, the loop-erasure
of the paths are the same from the first visit to $\p C_{k+2}$ onwards.
(See Lemma \ref{looperasefact} and Corollary \ref{extendcor}.)
Also, the paths are sufficiently separated,  and hence  using Lemma
\ref{june6.lemma1}, we can see that 
\[           \lambda(\bgamma_{k+ j+2}
\mid \bgamma_k) \geq c \, \mu(\bgamma_{k+ j+2}
\mid \bgamma_k) ,\]
and similarly for $\tilde \bgamma$.
\end{proof}

\begin{lemma}  There exists $c_0$ such that the following is true.
Suppose that $m+j \leq n \leq N-1$ and $\bgamma_n, \tilde \bgamma_n
\in \pairs_n$ with $\bgamma_n =_j \tilde \bgamma$ and $b(\bgamma_n)
\geq e^{-j/4} \,   b_n$.  Then we can couple $(\bgamma_{n+1},
\tilde \bgamma_{n+1})$ on the same probability space such
that, except perhaps on an event of probability at most $c_0 \, e^{-j/4}$,
\[                  \bgamma_{n+1} =_{j+1} \tilde \bgamma_{n+1}, \]
\[                (\bgamma_{n+1} \setminus \bgamma_n)
\cap C_{n-(j/2)} = \eset, \]
\[      b(\bgamma_{n+1}) \geq e^{-(j+1)/4} \,   b_n . \]
\end{lemma}

\begin{proof}  We write $\bgamma = (\gamma^1,\gamma^2),
\tilde \bgamma = (\tilde \gamma^1,\gamma^2)$.   In the
measure $\mu$, the conditional distribution of the remainder
of the paths is obtained by taking simple random walks starting
at the terminal point conditioned to avoid  the past and then
erasing loops.  Given the results so far,
we can couple random walks conditioned to avoid
$(\gamma^1,\gamma^2)$ and $(\tilde \gamma^1,
\tilde \gamma^2)$, except for an  event
of probability $O(e^{-j/2})$
they agree and stay in $\Z^d \setminus C_{n-(j/2)}$.
This will also be true of their loop erasures.

We claim that, except perhaps on this exceptional set,
\[   \frac{Q_{n+1}(\bgamma \oplus \bgamma')}{
Q_n(\bgamma )} = \frac{Q_{n+1}(\tilde \bgamma \oplus
\tilde  \bgamma')}{
Q_n(\tilde \bgamma )} \, [1 + O(e^{-j/2})].\] 
To see this we first see that
any loop $\ell$ in $C_{n+1}$ that   
intersects both
$\gamma^1 \oplus (\gamma')^1$ and $\gamma^2 \oplus (\gamma')^2$,
but does not intersect both
$\gamma^1 $ and $\gamma^2$ must intersect $\Z^3
\setminus C_{n-(j/2)}$.   If the loop does not also
intersect $C_{n-j}$, then this happens if and only if
$\ell$    
intersects both
$\tilde \gamma^1 \oplus (\tilde \gamma')^1$ and $\tilde \gamma^2 \oplus (\tilde
\gamma')^2$,
but does not intersect both
$\tilde \gamma^1 $ and $\tilde \gamma^2$.  The measure of the set of loops
that intersect $C_{n-j}$ and $\Z^3 \setminus  C_{n-(j/2)}$
is $O(e^{-j/2})$. f $d=2$, we only consider nondisconnecting loops.

\end{proof}

\begin{proof}[Proof of Proposition
\ref{coupling}]  Let $J_k = 0$.
Let $c_0$ be as in the previous lemma,
and choose $r \geq 3$ sufficiently large so that
\[                   \sum_{j=r-2}^\infty c_0 \, e^{-j/4} \leq \frac 12.\]
Choose $\epsilon$ so that for any $(\bgamma_i,\tilde \bgamma_{i})$ with
$i < M - r$, we can find a coupling $(\bgamma_{i+r},\tilde \bgamma_{i
+r})$ such that with probability at least $2\epsilon$ we have
$I(\bgamma_{i+r} ) = 1$,   $\bgamma_{i+r} = _{r-2} \, \tilde \bgamma_{i+r}$, and 
$b_{i+r}(\gamma_{i+r}) \geq e^{-(i+r)/4}$.
If the paths are coupled satisfying this we set $J_{k+1} = J_{k+2} = 0$
and $J_{k+2 +j} = j, j=1,2,\ldots,r$;   
otherwise, we set $J_{k+j} = 0, j=1,\ldots,r$.

Recursively, if we have seen $(\gamma_{k+jr},
\tilde \gamma_{k+jr},J_{k+jr})$.  We do the following.
\begin{itemize}
\item  If $J_{k+ jr} = 0$, we try to couple as above.  If
we succeed, then we set $J_{k+jr + 1} = J_{k+jr + 2}
=0$ and $J_{k+jr + s} = s-2$ for $s=3,4,\ldots,r$.
Note that for any $\bgamma_{k+jr}$, the probability
that we will be able to couple is 
at least $2\epsilon$.
\item If $J_{k + jr} = ir - 2 \geq r-2$, we couple as in
the last lemma for $r$ consecutive levels.  This will succeed
except for an event of probability at most 
\[  c_0 \sum_{t= ir-2}^{ (i+1)r-3} e^{-t/4}.\]
If we succeed, we set $J_{m+1} = J_{m} + 1$ for $m=k+jr,
\ldots,k+(j+1)r-1$; otherwise, we set $J_m = 0,
m=k +jr+1, \ldots,k + (j+1)r$. 
\end{itemize}

We now assume $N\geq k+2n$, and let
$\sigma = \min\{j \geq 1: J_{k+jr} = 0\}.$   From
the estimates above we see that 
$\Prob\{\sigma = \infty\} \geq \epsilon$ and
\[    \Prob\{\sigma = j \mid \sigma < \infty\}
 \leq c \, e^{-rj/4}.\]
  In particular,  we can find $ \alpha >0$ such
that  \[\E[e^{\alpha r \sigma}
 \mid \sigma  < \infty] \leq 1 +  \epsilon.\]
More generally if $\sigma_l =\min\{j > \sigma_{l-1}: J_{n+jr} = 0\},
$   then
\[  \Prob\{\sigma_l < \infty\} \leq (1-\epsilon)^l,\]
and \[\E[e^{\alpha r\sigma_l}; \sigma_l < \infty]
\leq \Prob\{\sigma_l < \infty\} \,
\E[e^{\alpha r\sigma_l} \mid \sigma_l < \infty] \leq
(1-\epsilon)^l \,  (1+\epsilon)^l = (1-\epsilon^2)
^l
.\]
In particular, if $\bar \sigma =
\max\{j : J_{k+jr}=0\}$,
then  
\[ \E[e^{\alpha r \bar \sigma}]  \leq \sum_{l=0}^\infty
\E[e^{\alpha r\sigma_l}; \sigma_l < \infty]  < \infty,\] and hence
\[  \Prob\{ r\bar \sigma \geq n
\}  \leq  e^{-\alpha n}  \E[e^{\alpha  r \bar \sigma}]
          \leq c\, e^{-\alpha  n}.\]

\end{proof}

The estimate \eqref{jun12.4} follows immediately
since
\[  \frac{\lambda_{n+1}}{\lambda_n}=  \sum_{\bgamma
  \in \pairs_n, \tilde \bgamma \in \pairs_{n+1}
      , \bgamma \prec \tilde \bgamma}   \lambda_n^\#(\bgamma) \, 
  \lambda_{n+1}(\tilde \bgamma \mid \bgamma),\]
 \[  \frac{ \lambda_{n+1}[\bfeta] }{ \lambda_n[\bfeta] }= \sum_{\bgamma
  \in \pairs_n, \tilde \bgamma \in \pairs_{n+1}
      , \bgamma \prec \tilde \bgamma}   \tilde \lambda_n^\#(\bgamma) \, 
  \lambda_{n+1}(\tilde \bgamma \mid \bgamma),
 \]
 where 
 $ \tilde \lambda_n^\#$ denote the
 normalized probability measure given $\bfeta$.
 The sums on the right-hand side are greater than some absolute $c_1 >0$ by the separation lemma.  Also,
 $\lambda_{n+1}(\tilde \bgamma \mid \gamma)
  \leq \mu_{n+1} (\tilde \bgamma \mid \gamma)$, so the sum over any set of paths $\bgamma$ of probability
 $O(e^{-\alpha j})$ is bounded by   $O(e^{-\alpha j})$.
 Finally, as we have seen, if $\gamma = _j \gamma'$,
 then
 \[      
      \left|
\sum_{  \tilde \bgamma \in \pairs_{n+1}
      , \bgamma \prec \tilde \bgamma}   
  \lambda_{n+1}(\tilde \bgamma \mid \bgamma)
   - \sum_{  \tilde \bgamma \in \pairs_{n+1}
      , \bgamma' \prec \tilde \bgamma}   
  \lambda_{n+1}(\tilde \bgamma \mid \bgamma')\right|
      \leq c \, e^{-\alpha j}.\]

\section{Proof of \eqref{jun12.3.alt}}

Suppose $A$ is a simply connected 
subset of $\Z^d$ containing the origin and
$x,y$
are distinct points in $\p A$, and, as before, 
$\hat A = A \setminus \{0\}$.
Let $\pairs(A;x,y)$
denote the set of SAWs starting at $x$, ending
at $y$, otherwise staying in $A$, and going
through  the origin.  To avoid trivial cases, we
assume that  $\pairs(A;x,y)$ is non-empty.
As in the case of $\pairs_{n}$, we can also
view elements of $\pairs(A;x,y) $
as ordered pairs $\bfeta = (\eta^1,\eta^2)
\in \saws_A^x \times \saws_A^y$ with
$\eta^1 \cap \eta^2 = \{0\}$.  Here $\saws_A^x$
denotes the set of SAWs starting at the origin,
ending at $x$, and otherwise staying in $A$.

The loop-erased measure on a single
path $\saws_A^x$ is given by
\[   \mu_{A,x}(\eta) =    
      (2d)^{-|\eta|} \,  F_\eta(\hat A).\]
This is the 
measure obtained by starting a simple
random walk at the origin,    stopping the path at the
first visit to $\p A$, restricting to the event that the
terminal vertex is $x$ and that there were no previous returns
to the origin, and then erasing loops.
      It has total mass $H_{\p \hat  A}(0,x)$. 
     (If we did not want to restrict to walks that do
not return to $0$ before leaving $A$, we would get
the same expression with an extra factor of $G_A(0,0)$;
this does not affect the probability distribution.)

\begad
If $A \supset C_{n+1}$ we can also view $\mu_{A,x}$ as
a measure on $\saws_n$ by considering the path stopped
at its first visit to $\p C_n$.   The measure $\mu_{A,x}$
is mutually absolutely continuous with respect to $\mu_n$
on $\saws_n$ (with constants uniform over all $n,x,A$ and
paths $\eta \in \saws_n$).  This would not be true  for 
 the measure on $\saws_n$ obtained by stopping a
simple random walk at $\p C_n $ and erasing loops.

\endad

We define the  loop-erased
measure $\lambda_{A,x,y}$
on $\pairs(A;x,y)$ to be absolutely continuous with
respect to  $\mu_{A,x} \times \mu_{A,y}$ on
$\saws_n^x \times \saws_n^y$ with Radon-Nikodym
derivative
\[ Y(\bfeta) =  Y_{A,x,y}(\bfeta) =     \exp \left\{-L_A(\bfeta)\right \}, \]
where $L_A(\bfeta)$ = $-\infty$ if $\bfeta
\not\in \pairs(A;x,y)$, and
\begin{itemize}
\item  ($d \geq 3$) $\log L_A(\bfeta)$ is the measure
of the set of loops in $\hat A$ that intersect $\eta^1$
and $\eta^2$.
\item  ($d=2$) $\log L_A(\bfeta)$ is the measure
of the set of loops in $\hat A$ that intersect $\eta^1$
and $\eta^2$ and do not disconnect $0$ from $\p A$.
\end{itemize}
If we
view the ordered pair $\bfeta$ as a single SAW $\eta$ from
$x$ to $y$, then this is the same as the weight
\begin{equation}  \label{sep25.1}
       (2d)^{-|\eta|} \, F_\eta(\hat A) \, \kappa_A, 
       \end{equation}
where $\kappa_A = 1$ for $d \geq 3$ and if $d=2$, $\log \kappa_A$
is the measure of loops in $\hat A$ that disconnect $0$ from $\p A$
(all of which intersect both $\eta^1$ and $\eta^2$).
We write $\lambda_{A,x,y}^\#$ for the corresponding probability
measure.  For  $d =2$,  the probability measure is the
same if we normalized the measure in \eqref{sep25.1}.

Another way to describe the probability
measure  $\lambda_{A,x,y}^\#$  is as follows.
\begin{itemize}
\item  Let $\omega^1$, $\omega^2$ be independent
conditioned random walks ($h$-process) where the walk
starts at $x,y$, respectively,  stops when it reaches the origin,
and is conditioned to reach the origin before returning to $\p A$.
\item Erase the loops from each walk separately and
reverse the paths  to
get $\bfeta = (\eta^1,\eta^2)$.
We can write this path as the pair $\bfeta$
or the single path $\eta = (\eta^1)^R \oplus \eta^2$.
\item Tilt the measure by $ \exp\{-L_A(\bfeta)\}/\rho_{A,x,y}$  
where
\[ \rho_{A,x,y} = \E\left[I_{\bfeta} \, \exp\{-L_A(\bfeta)\}\right] .\]
\end{itemize}

It follows from the definition that the probability
measure $\lambda_{A,x,y}^\#$ satisfies the following
``two-sided domain Markov property''.
\begin{itemize}
\item Suppose $\eta^1,\eta^2$ are disjoint SAWs starting
at $x,y$, respectively, and otherwise staying in $\hat A$,
with terminal vertices $x',y'$, respectively.   Then
in the measure $\lambda_{A,x,y}^\#$, conditioned
that the SAW has the form
\begin{equation}  \label{aug24.1}
\eta = \eta^1 \oplus \eta^* \oplus \eta_2^R , 
\end{equation}
the distribution of $\eta^*$
is $\lambda_{A \setminus (\eta^1 \cup \eta^2),x',
y'}^\#.$
\end{itemize}

Let us describe our strategy.   We will assume that $A$
is a simply connected subset with $  C_{n+1} \setminus A$ consisting of two
points $x,y \in \p_iC_{n+1}$.  If we prove our main result in this case, it
will hold more generally for $A \supset  C_{n+1}$ using the two-sided domain
Markov property by letting $\eta^1,\eta^2$ be the parts of the SAW
stopped at the first visit to $\p_iC_{n+1}$.
 Let us fix $k < n,$ 
$A$ a set as above with corresponding $x,y$,  $  \bfeta 
\in \pairs_k$,    and
to ease notation we will leave some dependence on these parameters
implicit.  However, all constants, including implicit
constants in $\asymp$ and $O(\cdot)$ notation will be uniform
over all choices.  Let $\pairs_A = \pairs(A;x,y)$, $\lambda_A = \lambda_{A,x,y}$,
$\lambda = \lambda_n$, $\tilde \lambda$ is $\lambda$
restricted to $\bgamma$ with $  \bfeta \prec \bgamma$,
and 
$\lambda^\#$ , $\tilde \lambda^\#$ the corresponding
probability measures obtained by normalization.  Let $\lambda_A^\#$
denote the probability measure on $\pairs_n$ obtained
from $\lambda_A$ by normalization and then truncating the
paths so they are in $\pairs_n$.  Proposition \ref{coupling}
states that we can define $(\bgamma,\tilde \bgamma)$
on the same probability space such that the marginal distribution
of $\bgamma$ is $\lambda^\#$, the marginal distribution of
$\tilde \bgamma$ is  $\tilde \lambda^\#$, and, except perhaps on
an event of probability $O(e^{-\alpha j})$ , we have
$\tilde \bgamma =_j \bgamma$ where $j = (n-k)/2$. 
Here $\alpha$ is an unknown positive constant which
we may assume is less than $1/4$.
 We let
\[     Z_A(\bgamma) = \frac{\lambda_A^\#(\bgamma)}{
  \lambda^\#(\bgamma)}, \;\;\;\; \bgamma \in \pairs_n.\]
The main work will be to establish the following
proposition.

  \begin{proposition}  \label{sep9.prop1}
   There exist $   c_2 < \infty$
such that if  $A,x,y$ are as above,  then for all
$\bgamma \in \pairs_n$,
\[ Z_A(\bgamma) 
\leq c_2 . \]
Moreover, if $\bgamma =_j \tilde \bgamma$,
\[              |Z_A(\bgamma) - Z_A(\tilde \bgamma)|
  \leq c_2 \, e^{-j/4}.\]
\end{proposition}

\begad
If $\bgamma \in \Sep_n$, then $Z_A(\bgamma) \asymp 1$.  However,
if the tips of $\gamma^1,\gamma^2$ are close, it is possible for
$Z_A(\bgamma)$ to be small.
\endad

The theorem follows almost immediately from
the proposition as we now show.
 Let $q$ denote a probability measure
on $\pairs_n \times \pairs_n$ such that the marginal
distributions are $\tilde \lambda^\#,\lambda^\#$,
respectively and such that
\[     q\{(\tilde \bgamma, \bgamma):
\tilde \bgamma =_j \bgamma\} \geq 1 -
 c\, e^{-j\alpha}.\]
Note that  $\tilde \lambda^\#$ is the same as
$[\lambda^\#[\pairs_n(\bfeta)]] ^{-1}
\, \lambda^\#$, restricted to $\pairs_n(\bfeta)$,
and hence  of $Z = Z_A,$
\[ \sum_{\bfeta \prec \bgamma}
  \lambda_A^\#(\bgamma) = 
  \sum_{\bfeta \prec \bgamma}
  \lambda^\#(\bgamma) \, Z(\bgamma)
     =  \lambda^\#[\pairs_n(\bfeta)]
      \sum_{\bgamma}
       \, \tilde \lambda^\#(\bgamma)
        \, Z(\bgamma).\]
Also, given Proposition \ref{sep9.prop1},
\begin{eqnarray*}
\sum_{\tilde \bgamma}
       \, \tilde \lambda^\#(\tilde \bgamma)
        \, Z(\tilde \bgamma)
        & = & 
\sum_{(\tilde \bgamma,\bgamma)}
       \, q(\tilde\bgamma,\bgamma)
        \, Z(\tilde \bgamma)\\
        & = & O(e^{-j\alpha})+  \sum_{(\tilde \bgamma,\bgamma)}
       \, q(\tilde\bgamma,\bgamma)
        \, Z( \bgamma)\\
    & = &  O(e^{-j\alpha})+  \sum_{\bgamma}
       \, \lambda^\#(\bgamma)
        \, Z( \bgamma)\\
        &=& O(e^{-j\alpha})+  \sum_{\bgamma}
       \, \lambda^\#_A(\bgamma)
        = 1 +  O(e^{-j\alpha}).
        \end{eqnarray*}
   Therefore,
    \[ \sum_{\bfeta \prec \bgamma}
  \lambda_A^\#(\bgamma) =  \lambda^\#[\pairs_n(\bfeta)]
    \,  \left[1 +  O(e^{-j\alpha})\right].\]

\begin{proof}[Proof of Proposition \ref{sep9.prop1}]
We will first prove the result in the case where
$x$ and $y$ are separated, say  $|x-y| \geq
e^{n-5}$.
We write each $\hat \gamma^1 \in \saws_{A}^x$ as
\[  \hat \gamma^1 = \gamma^1 \oplus \eta^1,\]
where $\gamma^1 \in \saws_n$. We do similarly for $\hat \gamma^2
\in \saws_A^y$ and we also write
\[   \hat \bgamma = \bgamma \oplus   \bfeta.\]
We write
\[   L_A(\hat \bgamma) = L_n(\bgamma)
  + \tilde L (\hat \bgamma), \]
 where $\tilde L (\hat \bgamma) = 
  L _A(\hat \bgamma) - L_n(\bgamma).$
   If $d=2$, we restrict to nondisconnecting loops.  We will
   also write $L_n(\bgamma) = -\infty$
   if $\gamma^1 \cap \gamma^2 \neq \{0\}$,
    and $\tilde L(\hat \gamma) = - \infty$
   if $\hat \gamma^1 \cap \hat \gamma^2 
   \neq \{0\}$. 

   We write $\mu^\#$ for the measure on loop-erased
   walks obtained by taking an infinite loop-erased walk
   and stopping it at the first visit to $\p C_{n}$.
   We write $\mu_A^\#$ for the measure obtained from
   the loop-erasure of a random walk from $0$ conditioned
   to leave $A$ at $x$ or $y$ (that is, we stop the
   walk at $\p A$ and the erase the loops).

As we have seen, we have
 \[    \lambda^\#(\bgamma) =a_n\, 
    \mu^\#(\bgamma) \, \exp\{-L_n(\bgamma)\} , \]
  where
  \[   a_n^{-1} = \E\left[   \exp\{-L_n\} \right] , \]
 and the expectation is with respect to   $\mu^\#$. Similarly,
\begin{equation}  \label{sep8.2.alt}
  \lambda_A^\#(\bgamma)
   =b_{A} \sum_{     \bfeta } \mu_A^\#(\bgamma)
    \, 
    \mu_A^\#( \bfeta \mid  \bgamma ) \, \exp\left\{-[L_n(\bgamma
    )   +
         \tilde L (\bgamma \oplus \bfeta) ]\right\},
         \end{equation}
         where
  $\log          \tilde L (\bgamma \oplus \bfeta)$
  is the measure of loops in $A$ that intersect both $\gamma^1\oplus \eta^1$ and $\gamma^2 \oplus \eta^2$ but are not loops in $C_n$ intersecting
  $\gamma^1$ and $\gamma^2$, and
  \[     b_A^{-1 } = \E_A\left[  \exp\left\{-(L_n +
  \tilde L 
     ) \right\} \right] , \]
     where the expectation is with respect to
      $\mu_A^\#$.    Therefore,
  \[ Z_A(\bgamma)
    = \frac{a_n}{b_A} \, \frac{\mu_A^\#(\gamma^1)
     \, \mu_A^\#(\gamma^2)}
       {\mu_n^\#(\gamma^1)\,\mu_n^\#(\gamma^2) } \, 
       \sum_{     \bfeta }
    \mu_A^\#( \bfeta \mid  \bgamma ) \, \exp\left\{- 
         \tilde L ( \bgamma \oplus \bfeta) ]\right\}.\]

Recall that
\[  \frac{\mu_A^\#(\gamma^1)
  }
       {\mu_n^\#(\gamma^1)  }
        = \frac{H_{\p(A \setminus \gamma^1)}(z,x)}
           {H_{\p \hat A}(0,x)\, \Es_{\gamma^1}(z) \, \phi_A\,  \phi_A(\gamma^1)}, \]
       where
       \begin{itemize}
     \item $(d \geq 3)$ $\phi_A = 1$
     and  $\log \phi_A(\gamma^1)$ is the measure of
   loops in $\hat \Z^d$ that intersect $\gamma^1$
 but do not lie in $A$. 
 \item $(d=2)$   $\log \phi_A$ is the measure of loops in $\hat \Z^2$ that do not lie in $A$ but disconnect $0$ from
 $\p C_n$ and $\log  \phi_A(\gamma^1)$ is
 the measure of nondisconnecting loops not
 in $A$ that intersect $\gamma^1$.
\end{itemize}
Lemma \ref{sep14.lemma1} shows that $\phi_A \asymp n$ if $d=2$, and hence
we can use Corollaries \ref{cor1}, \ref{cor3} to conclude that 
\[ H_{\p(A \setminus \gamma^1)}(z,x) \asymp
     H_{\partial A}(0,x) \,  {\Es_{\gamma^1}(z) \, \phi_A}, \;\;\;\;
      \phi_A(\gamma^1) \asymp 1, \]and using Corollary \ref{cor4.1},
      if $\gamma^1 =_j \tilde \gamma^1$,
 \[    \frac{\mu_A^\#(\gamma^1)
  }
       {\mu_n^\#(\gamma^1)  }   =
        \frac{\mu_A^\#(\tilde \gamma^1)
  }
       {\mu_n^\#(\tilde \gamma^1)  }
        [1 + O(e^{-j/4})].\]
The same results hold for $\gamma^2$, and hence
if $\bgamma=_j \tilde \bgamma $,
\begin{equation}  \label{sep26.2}
    \frac{\mu_A^\#(\bgamma )
  }
       {\mu_n^\#(\bgamma )  }   =
        \frac{\mu_A^\#(\tilde \bgamma )
  }
       {\mu_n^\#(\tilde \bgamma )  }
        [1 + O(e^{-j/4})].
        \end{equation}

Since, 
\[   \sum_{     \bfeta }
    \mu_A^\#( \bfeta \mid  \bgamma ) \, \exp\left\{- 
         \tilde L ( \bgamma \oplus \bfeta) ]\right\} \leq 1, \]
         we see that $Z_A $ is uniformly bounded above.
If $\bgamma =_j \tilde \bgamma$,
then
\[ 
 \left|\sum_{     \bfeta }
   \mu_A^\#( \bfeta \mid  \bgamma ) \, \exp\left\{- 
         \tilde L ( \bgamma \oplus \bfeta)  \right\}
          -
          \sum_{     \bfeta }
    \mu_A^\#( \bfeta \mid  \tilde \bgamma ) \, \exp\left\{- 
         \tilde L ( \tilde \bgamma \oplus \bfeta)  \right\}\right|
       \hspace{.3in}  \]
\[ \leq   \sum_{     \bfeta }
   |\mu_A^\#( \bfeta \mid  \bgamma )-
    \mu_A^\#( \bfeta \mid  \tilde \bgamma )| \hspace{2in} \]
   \[ \hspace{1in}    
      +   \sum_{     \bfeta }  \mu_A^\#( \bfeta \mid  \bgamma )  \, \left|\exp\left\{- 
         \tilde L (   \bgamma \oplus \bfeta)\right\}
           - \exp\left\{- 
         \tilde L ( \tilde \bgamma \oplus \bfeta)\right\}\right| 
     .\]
 We use Corollary \ref{cor4.1} to see that
 \[  \sum_{     \bfeta }
   |\mu_A^\#( \bfeta \mid  \bgamma ) -
    \mu_A^\#( \bfeta \mid  \tilde \bgamma )| \leq  O(e^{-j}), \]
  \[   \sum_{     \bfeta  \cap C_{n-(j/2)} \neq \eset}  \mu_A^\#( \bfeta \mid  \bgamma ) 
   \leq O(e^{-j}), \]
and, again, if  $\bfeta  \cap C_{n-(j/2)} =\eset$,
\[  \left|\exp\left\{- 
         \tilde L (   \bgamma \oplus \bfeta)\right\}
          - \exp\left\{- 
         \tilde L ( \tilde \bgamma \oplus \bfeta)\right\}\right| 
           \leq O(e^{-j/2}).\]
 This establishes the proof when $x$ and $y$ are separated.

 \begad
 
 The proof when $x$ and $y$ are separated is the main part.  If $x$ and $y$ are
 not separated, then we first separate $x$ and $y$ and then use the argument for
 separated points.  This is straightforward using a separation lemma, but we discuss
 the details below.
 
 \endad

When $x$ and $y$
are   not separated, we first run   paths from $x$ and $y$  
until the paths get to $C_{n+(4/5)}$ and use the fact that there
is a positive probability that the paths have separated.
Indeed, we write each $\hat \gamma^1 \in \saws_{A}^x$ as
\[  \hat \gamma^1 = \gamma^1 \oplus \hat \eta^1 \oplus \eta^1,\]
where $\gamma^1 \in \saws_n$ and $\eta^1$ starts at the last visit
to $C_{n+ (4/5)}$. We do similarly for $\hat \gamma^2
\in \saws_A^y$ and we also write
\[   \hat \bgamma = \bgamma \oplus \hat \bfeta \oplus \bfeta.\]
We partition loops that intersect both $\hat \gamma^1$ and $\hat \gamma^2$
into three sets:
\begin{itemize}
\item  Loops in $C_n$ that intersect both $\gamma^1$ and $\gamma^2$.
\item  Loops in $A$ that intersect both $\eta^1$ and $\eta^2$
\item  All other loops.
\end{itemize}
If $d=2$, we restrict to nondisconnecting loops,
In that way we write
\[    L_A(\hat \bgamma) = 
    L_n( \bgamma) + L_A(\bfeta) + \tilde L(\hat \bgamma).\]
   As before, we  
  write $L_n(\bgamma) = -\infty$
   if $\gamma^1 \cap \gamma^2 \neq \{0\}$,
   $L_A(\bfeta) = - \infty$ if $\eta^1 \cap \eta^2
 \ne \eset$, and $\tilde L(\hat \gamma) = - \infty$
   if $\hat \gamma^1 \cap \hat \gamma^2 
   \neq \{0\}$. 

   We write $\mu^\#$ for the measure on loop-erased
   walks obtained by taking an infinite loop-erased walk
   and stopping it at the first visit to $\p C_{n}$.
   We write $\mu_A^\#$ for the measure obtained from
   the loop-erasure of a random walk from $0$ conditioned
   to leave $A$ at $x$ or $y$ (that is, we stop the
   walk at $\p A$ and the erase the loops).

By definition, we have
 \[    \lambda^\#(\bgamma) =a_n\, 
    \mu^\#(\bgamma) \,   \exp\{-L_n(\bgamma)\} , \]
  where
  \[   a_n^{-1} = \E\left[   \exp\{-L_n\} \right] , \]
 and the expectation is with respect to   $\mu^\#$. Similarly,
\begin{equation}  \label{sep8.2}
  \lambda_A^\#(\bgamma)
   =b_{A} \sum_{     \bfeta \oplus \hat \bgamma}
    \mu_A^\#( \bgamma \oplus \bfeta \oplus \hat \bgamma) \, \exp\left\{-[L_n(\bgamma
    ) + L_A(\bfeta) +
         \tilde L(\tilde \bgamma) ]\right\},
         \end{equation}
  where
  \[     b_A^{-1 } = \E_A\left[  \exp\left\{-(L_n + L_A +
         \tilde L) \right\} \right] , \]
      and the expectation is with respect to
      $\mu_A^\#$.   

      We will choose from
      $\mu_A^\#$ by choosing $\bfeta$ first, then
      $\bgamma$, and then finally $\hat \bfeta$.
     By doing this we can see that  \eqref{sep8.2} can be written as
      \[
      \lambda_A^\#(\bgamma)
      =b_A \sum_{\bfeta}\sum_{\hat \bfeta}
            \mu_A^\#(\bfeta) \, \mu_A^\#(\bgamma \mid
              \bfeta) \, \mu_A^\#(\hat\bfeta
               \mid \bgamma, \bfeta)  \,
                \exp\left\{-[L_n(\bgamma
    ) + L_A(\bfeta) +
         \tilde L(\tilde \bgamma) ]\right\},
               \]
which in turn can be written as  
\[  b_A\sum_{\bfeta} \mu_A^\#(\bfeta) \, \exp\{-L_A(\bfeta)\}
     \, \mu_A^\#(\bgamma \mid \bfeta) \, \exp\{-L_n(\bgamma)\}
     \, \Psi(\bgamma, \bfeta), \]
 where
 \[ 
      \Psi(\bgamma, \bfeta)  = \sum_{\hat \bfeta} \mu_A^\#(\hat \bfeta \mid
          \bgamma,\bfeta) \, \exp\{-\tilde L(\bgamma \oplus
            \hat \bfeta \oplus \bfeta)\}
             \leq 1.\]
Therefore,
\[  \frac{\lambda_A^\#(\bgamma)}{\lambda^\#(\bgamma)}
  = \frac{b_A}{a_n }
   \, \sum_{\bfeta} \  \mu_A^\#(\bfeta) \, \exp\{-L_A(\bfeta)\}
     \,\frac{\mu_A^\#(\bgamma \mid \bfeta)}{\mu^\#(\bgamma)}
       \,  \Psi(\bgamma, \bfeta) .\]

We will need to use the separation lemma on both the
beginning and the ending of the path.  There is a lot
of arbitrariness in the definition of the separation event.
We will not be specific here, but the important facts are
that there there exists $c_1 > 0$ such that 
\begin{equation}  \label{sep8.4}
\Psi(\bgamma,\bfeta) \geq c_1, \;\;\;\;
\bgamma,\bfeta \in \Sep   ,
\end{equation}
\[  \sum_{\bfeta
\in \Sep } \mu_A^\#(\bfeta) \, \exp\{-L_A(\bfeta)\}
\geq c_1 \, \sum_{\bfeta
 } \mu_A^\#(\bfeta) \, \exp\{-L_A(\bfeta)\},\]
\[  \sum_{\bgamma^* \in\Sep  }
      \mu_n^\#(\bgamma^*) \, \exp\{-L_n(\bgamma^*)\}
     \geq c_1 \, 
      \sum_{\bgamma^* }
      \mu_n^\#(\bgamma^*) \, \exp\{-L_n(\bgamma^*)\}.\]
To be specific, we will say that $\bfeta \in \Sep$ if each
initial point (the vertex in $\p C_{n + (4/5)}$) is distance
at least $e^n/100$ from the other path in the pair.  For
$\bgamma^* \in \pairs_n$   we can use the definition of
separation as before.   Establishing \eqref{sep8.4} uses
the same argument as in Proposition \ref{sep26.1}  ---  given $(\bgamma,\bfeta)$
that are separated, the conditioned simple random walks
whose loop erasure given $\hat \bfeta$ have a positive
probability of staying apart; hence so do the paths in
$\hat \bfeta$; and the measure of loops intersecting both
is bounded uniformly above zero.

We now claim the following.
\begin{itemize}
\item If $\bgamma=_j \tilde \bgamma$, then
\begin{equation}  \label{sep9.1}
\left|\frac{\mu_A^\#(\bgamma \mid \bfeta)}{\mu^\#(\bgamma)} - \frac{\mu_A^\#(\tilde \bgamma \mid \bfeta)}{\mu^\#(\tilde \bgamma)}\right|
 \leq c \, e^{-j/4} . 
 \end{equation}
\end{itemize}

We will consider the result without conditioning,
\[ \left|\frac{\mu_A^\#(\bgamma  )}{\mu^\#(\bgamma)} - \frac{\mu_A^\#(\tilde \bgamma )}{\mu^\#(\tilde \bgamma)}\right|\]
where the estimate should be uniform over all $A$
containing  $C_{n+(4/5)} \setminus \p_i C_{n+(4/5)}$
and boundary points $x,y \in C_{n +(4/5)}  \cap 
  \p A$.
and all $x,y \in \p A$.  Since
$\mu_A^\#(\gamma \mid \bfeta) = \mu_{A \setminus \bfeta}^\#
(\gamma)$ (with appropriately chosen starting points), we 
get our estimate as in \eqref{sep26.2}.

We have already noted that $\Psi(\bgamma,\bfeta)
\leq 1$.  We now claim the following.
\begin{itemize}
\item If $\bgamma =_j \tilde \bgamma$, then
\[    |\Psi(\bgamma, \bfeta)
 - \Psi(\tilde \bgamma, \bfeta)| \leq c \, e^{-j/4}.\]
\end{itemize}

To show  this we first  note again    that
\[   \sum_{\bfeta^* \cap C_{n-(j/2)} \ne \eset}
  \mu^\#_A(\bfeta^* \mid \bgamma,\bfeta)
     \leq c \, e^{-j/4}.\]
     Also, if $\bfeta^* \cap C_{n-(j/2)}  =\eset$,
    we have
    \[     \bgamma \oplus \bfeta^* \oplus \bfeta
       =_{j/2} \, \tilde \bgamma \oplus \bfeta^* \oplus
         \bfeta, \]
   in which case
  \[  |\tilde L(   \bgamma \oplus \bfeta^* \oplus
         \bfeta) - \tilde L( \tilde \bgamma \oplus \bfeta^* \oplus
         \bfeta)| \leq c \, e^{-j/4}.\]

   \end{proof}

   \appendix

   \section{On the proof of Lemma \ref{importantlemma}}
\label{rwproofsec}

The proof of the result for $d =2$ was done in and
it was adapted for $d=3$ in \cite{Shir}.   We will not give the
complete proof here, but we will sketch most of the argument
using a   different function than used before. 
Let us first consider Brownian motion, where the result is easier.

Let $U$ denote the open cube
\[   U = \left\{x=(x^1,\ldots,x^d) \in \R^d: |x^j - 1|
< 1\right\}, \]
and let $V = \p U \cap \{x^1 = 1\}$.  Let $g(x)$
be the the harmonic function on $U$ with boundary
value $1_V$, that is, $g(x)$ is the probability that a
Brownian motion starting at $x$ exits $U$ at $V$.
By symmetry, we see that $g(0) = 1/2d$.  Let
\[  U^- = \{(x^1,\ldots,x^d) \in U: x^1 \leq 0\},\]
and let $L$ be the line segment $  
L = \{(x^1,0,\ldots,0) \in U : 0 < x^1 < 1\}.$
We will need several properties of  $g$.

\begin{lemma}$\;$
\begin{enumerate}
\item   There exists $c < \infty$ such
that $ |x|/4d \leq g(x) -  g(0) \leq c \, |x|$ for $x \in L$.
\item  There exists $c_1 >0$ such that if  $x = (x^1,\tilde x)
\in U^-$, then $g(x) \leq g(0) - c_1 |\tilde x|^2.$
\item   There exists $\epsilon > 0$ such that if
$U_\epsilon^- =  \{x = (x^1,\tilde x)
\in U: x^1 \leq \epsilon |\tilde x|^2\},$  
then
\begin{equation}  \label{jun12.1}
 g(0) = \sup\{g(x): x \in U_\epsilon^- \}.
 \end{equation}

\end{enumerate}

\end{lemma}

\begin{proof}$\;$

\begin{enumerate}

\item  
The upper bound follows from derivative estimates for
harmonic functions, so we will only show the lower bound
using
a coupling argument. Let $B_t$ be a Brownian
motion starting at the origin and $W_t = B_t +x$
a Brownian motion starting at $x$.  Let
$T = \min\{t: B_t \in \p U\}, T^x = \min\{t: W_t
\in \p U\}$.  Note that $W_t - B_t  = x$ for all $t$.
\begin{itemize}
\item  if $T = T^x$, then $B_\tau$ and $W_\tau$
cannot be exiting at $V$.
\item  If $ T < T^x$, then $B_{T} \not \in V$.
There is still a chance that $B_{T^x} \in V$.
\item  If $T^x < T$, then $W_{T^x} \in V$ and
$B_{T^x} = W_{T^x} - x$. By symmetry, we can see that
\[    \Prob\{T^x < T\}
\geq \Prob\{B_T \in V\} = \frac 1{2d},\]
Using the gambler's ruin
on the first component, we can see that, given
$T^x < T$,  the probability
that the first component of $B_{t}$ will equal 
$-1$ before it equals $1$ is  $|x|/2$. In particular,
\[     \Prob\left\{B_{T} \not\in V \mid B_{T^x} \in V
\right\}
 \geq \frac{|x|}{2}.\]
Therefore,
\[   g(x) - g(0) = \Prob\left\{B_{T^x} \in V, B_T \not\in V\right\} \geq 
\frac{|x|}{4d}.\]
\end{itemize}

\item
The previous argument can be used to show that
$g(-\delta,x^2,\ldots,x^d)  < g(0,x^2,\ldots,x^d)$
for $\delta > 0$ so it suffices to consider the
maximum over $(0,x^2,\ldots,x^d)$.  By symmetry, we
can consider the same problem for the positive
quadrant $Q = \{(x^1,\ldots,x^d) \in U: x^j > 0\}$
where we reflect the Brownian motion on the boundaries
$\{x^j = 0\}$.   
We can couple reflected
Brownian motions $B_t$ starting at $0$
and $W_t$  at $x \in \{(0,x^2,\ldots,x^d):
0 \leq x^j < 1\}$, so that
all of the components of the latter Brownian motion are greater
than those of the former Brownian motion and the first
components always agree.   This gives $d$ independent
coupled reflecting Brownian motions (the coupling
is described in the next paragraph) with
the following properties: the first components  are
the same and start at $0$; for the others, one starts at $0$,
the other at $x^j$, and coupling occurs when the latter reaches
$x^j/2$.  In particular, the probability that the latter one
reaches $1$ before coupling is greater
than  $x^j/2$, and given that, the probability that
the first one exits somewhere other than  $V$ is greater
than $c|x^j|$.
Using this we can see that
\[        \Prob\{T^x < T^0\} \geq c \, |x|^2.\]

We now describe the
coupling of one-dimensional
reflected Brownian motions
$(Y_t,Z_t)$   with
$Y_0 = 0, Z_0 = x \in (0,1)$.
Let  $W_t$ be a Brownian motion starting
at $x$ and let $Z_t = |W_t|$.  Let
$\sigma_{x/2},\sigma_1$  be the first
times that $Z_t=x/2, Z_t = 1$, respectively.  Let
$Y_t = |W_t - x|$ for $t \leq \sigma_{x/2}$ and $Y_t = Z_t$
for $t \geq \sigma_{x/2}$, and note that $Y_t$ is
a reflected Brownian motion starting at the origin.
Let $T_Z,T_Y$ denote the first time
greater than or equal to  $\sigma_{x/2} \wedge
\sigma_1$ at which $Z_t \in \{0,1\}, Y_t \in \{0,1\}$, respectively.
\begin{itemize}
\item  If $\sigma_{x/2} < \sigma_1$, then  $T_Z= T_Y$.
\item  If $\sigma_1 < \sigma_{x/2}$, $T_Z = \sigma_1$
and then $Y_{\sigma_1} = 1-x$.  Therefore, using
the gambler's ruin estimate,
\[    \Prob\{Y_{T_Y} = 0 \mid \sigma_1 < \sigma_{x/2}
\} = x.\]
\end{itemize}
Therefore, 
\[    \Prob\{Y_{T_Y}=0\} = \Prob\{\sigma_1 < \sigma_{x/2}\}
 \,  \Prob\{Y_{T_Y} = 0 \mid \sigma_1 < \sigma_{x/2} \}
 \geq \frac{x^2}{2}. \]

\item  This follows from Part 2 and derivative estimates
for positive harmonic functions.

\end{enumerate}
\end{proof}

Now suppose $A \subset \{x \in \R^d:|x| \geq 1\}$ and $|y| = 1$.  Let $ U'$ be
a rotated, dilated, and translated version of the set $U$ above
centered at $y$, rotated so that the
inward radial direction of $U'$ corresponds to the positive first
component, and dilated by a factor $\delta$ where $\delta$ is sufficiently
small so that the analogue of $U_\epsilon^-$ lies entirely in $\{|x|
\geq 1\}.$  Using \ref{jun12.1},
we can find a uniform $\delta$ depending only on the
$\epsilon$ in \eqref{jun12.1}. Let $V'$ be the analogue of $V$ and
let $f$ be the analogue of $g$, that is, $f(x)$ is the probability
that a Brownian motion starting at $x$ exits $U'$ at $V'$.  There
exists $\delta'> 0$ such that all the point in $V'$ have radius
at most $1-\delta'$.   Let $\tau_A$ be the first time a Brownian motion
hits $A$.  Then if $r < 1$ and $z = ry \in U'$, we have $g(z) \geq g(x)$
for all $x \in A \cap U'$, and hence,
\[           \Prob^z\{B_{T} \in V \mid T < \tau_A\}
\geq \Prob^z\{B_T \in V\} \geq \frac 1{2d},\]
\[   \Prob^z\{B_{T} \in V \mid T >  \tau_A\}
\leq \Prob^z\{B_T \in V\} .\]

For random walk, we need to find the discrete analogue of
the function $g$.  Some work needs to be done because
we do not want to require  that  the cube $U$  be lined
up with the lattice $\Z^d$; indeed, we want an estimate that
is  uniform over all rotations. Let $U'$ be a rotation of
the $U$ above, with corresponding $V \subset \p U'$, and
$g$ the harmonic function on $U$ with boundary value $1_V$.
For each $r =1/n > 0$,   Let $K_r =
\{x \in \Z^d: rx \in U'\}$.  Let $V_r$ be the
subset of $\p K_r$ corresponding to $nV$.  
Let $g^*(x)$ be the discrete harmonic function on $K_r$ with
boundary value $1_{V_r}$, that is, $g^*(x)$ is the probability
that a random walk starting at $x$ exits $K_r$ at $V_r$.

\begin{lemma}
There exists $c < \infty$ such that
for every rotation and every $n$,
\[  \left|g^*(x) - g(rx) \right| \leq \frac c n  \;\;\;
\mbox{ if } \;\;\;
\dist(x,\p V_r) \geq  \frac{n}{10}.\]
\end{lemma}

\begin{proof}  We use $\Delta$ to denote the discrete
Laplacian, let $\hat g(x) = g(rx)$ and recall that
\[       \hat g(x) = g^*(x) -  \sum_{y \in U_r}
 G_{U_r}(x,y) \, \Delta \hat g(y) , \]
where $\Delta$ denotes the discrete Laplacian.  Therefore,
it suffices to show that \begin{equation}  \label{jun13.1}
    \sum_{y \in U_r}
 G_{U_r}(x,y) \, |\Delta \hat g(y)| \leq  \frac c n
 , \;\;\;\;
\dist(x,\p V_r) \geq  \frac{n}{10}, 
\end{equation}
and to prove \eqref{jun13.1}  it suffices to establish that
\begin{equation}   \label{jun13.2}
\sum_{y \in U_r, \; j-1 < \dist(y,\p U_r)
\leq j} G_{U_r}(x,y) \, |\Delta \hat g(y)|
  \leq     \frac{c}{  j^2n}.
  \end{equation}

We will use the following fact that uses only the Taylor
approximation of $h$ and the derivative bounds for
harmonic functions.
\begin{itemize}
\item There exists $c > 0$ such that if $R \geq 2$
and  $h$ is a (continuous) harmonic function on $\{x \in \R^d: |x| < R \}$,
then
\[     \Delta h(0) \leq \frac{c}{R^4} \, \sup_{|x| \leq R/2}
|h(x) - h(0)| ,\]
where $\Delta$ denotes the discrete Laplacian.
\end{itemize}

We now split into two cases.
\begin{itemize}
\item  Suppose that $\dist(y, V_r) \geq n/20$.  In this
case, the gambler's ruin estimate shows that 
\[ \hat g(y)
\leq c \, \frac{\dist(y,\p U)}{n} . \]
Therefore, if $j-1 < \dist(y,\p U) \leq j$,
\[   \left|\Delta \hat g(y) \right| \leq c \, j^{-4}  \, \frac{j}{n}
= \frac{c}{j^3n}.\]
Also, using the gambler's ruin estimate, we can see that
for all $x$, 
\[    \sum_{j-1 < \dist(y,\p U_r)
\leq j} G(x,y)  \leq c \, j .\]
Therefore, 
\[     \sum_{y \in U_r, \; j-1 < \dist(y,\p U_r)
\leq j, \dist(y, V_r) \geq n/20} G_{U_r}(x,y) \, |\Delta \hat g(y)|
  \leq     \frac{c}{  j^2\,n}.\]
\item  Suppose that $\dist(y, V_r) = \delta \leq  n/20$
and let $\delta' = \dist(y,\p U_r \setminus V_r). $
\begin{itemize}
\item  Suppose that $\delta < \delta'$.  Then using
the gambler's ruin estimate, we can see that there exists $c$
such that for $|z-y| < \delta/2$,
\[               1 - \hat g(z) \leq   c \, \delta/\delta'.\]
Therefore,
\[         |\Delta \hat g(y)| \leq c \, \delta^{-4} \, (\delta/\delta') =
\frac{c}{\delta^3 \, \delta'}.\]
\item  Suppose that $\delta ' \leq \delta$.  
Then similarly that there exists $c$
such that for $|z-y| < \delta/2$,
\[                \hat    g(z) \leq   c \, \delta'/\delta .\]
Therefore,
\[         |\Delta \hat g(y)| \leq  \frac{c}{\delta \, \delta'^3}.\]

\end{itemize}
Using the gambler's ruin estimate (or similarly), we see that
\[       G_{U_r}(x,y) \leq \frac{c}{n^{d-2}} \, \frac{\delta}{n}
\, \frac{\delta'}{n} = \frac{c \, \delta\, \delta'}{n^d},\]
and hence if $j-1 < \dist(y,\p U_r)
\leq j $,  
\[     G_{U_r}(x,y)\,  |\Delta \hat
g(y)| \leq \frac{c}{(\delta \wedge \delta')^2\, n^d}
\leq \frac{1}{j^2  \, n^d} .\]
Using the fact that $\#\{x \in U_r:  j-1 < \dist(y,\p U_r)
\leq j\}  \leq c \, n^{d-1}, $
we get that 
\[     \sum_{y \in U_r, \; j-1 < \dist(y,\p U_r)
\leq j, \dist(y, V_r) \leq n/20} G_{U_r}(x,y) \, |\Delta \hat g(y)|
  \leq     \frac{c}{  j^2\,n}.\]  
\end{itemize}
This establishes \eqref{jun13.2}, and hence proves the lemma.
\end{proof}

Using this we establish that there exists a $c_1$ such
that there exists $x$ with $|x| \leq c_1$ and
\[       g^*(x) \geq g^*(y), \;\;\;\;
 y \in n \, U_\epsilon^-.\]
 Therefore, for any $A \subset \Z^d \setminus C_n$,
 the probability that random walk starting at $x$
 leaves on the analogue of $V'$ given that if avoids
 $A$ is greater than the unconditioned probability which
is greater than $(1/2d)$.

\section{Sketch of proof of Separation Lemma I}

If $\bfeta = (\eta^1,\eta^2)
\in \pairs_n$, with terminal vertices $z_1,z_2$,
we define the separation $\Delta(\bfeta) = \Delta_n(\bfeta)$ to
be the largest $r$ such that 
\[   \dist(z _j, \eta^{3-j})  \geq r\,  e^{n } ,\;\;\;\;j=1,2,
.\]

\begin{itemize} 

\item {\bf Claim.}  There exists $r_0, c > 0$ such that if
$0< r < r_0$ and 
$\bfeta \in \pairs_n$ with $\Delta(\bfeta) \geq r $, then
\[    \sum   \lambda_{n+8r}(\bfeta' \mid \bfeta) \geq c .\]
where the sum is over all $\bfeta' = (\eta^1 \oplus
\tilde \eta^1, \eta^2 \oplus \tilde \eta^2) \in  
\pairs_{n+8r}$ with
\[    \diam[\tilde \eta^i] \leq 16 r \, e^n , \;\;\; i=1,2,\]
\[  \dist(\tilde \eta^i,\eta^{3-i} \oplus \tilde \eta^{3-i})
\geq \frac r2 \, e^{n},\;\;\;\; i=1,2,\]
\[   \Delta(\bfeta') \geq 2r.\]

\end{itemize}

Indeed, using Proposition \ref{basiclerw} and Lemma \ref{importantlemma},
we can see that
\[    \sum   \mu(\bfeta' \mid \bfeta) \geq c .\]
Also, any loop that intersects both $\eta^1 \oplus
\tilde \eta^1$  and $ \eta^2 \oplus \tilde \eta^2$
but does not intersect both $\eta^1$ and $\eta^2$,
must be of diameter at least $ (r/2) \, e^{n}$
and intersect either $\tilde \eta^1$  or $ \tilde \eta^2$.
Since the diameters of these curves are bounded by
$16re^{n}$ we see (Lemma \ref{june6.lemma1})
that the loop measure of such curves
is uniformly bounded.  Hence $Q_{n+8r}(\bfeta') \geq 
c\, Q_{n}(\bfeta).$

\begin{itemize}  
\item {\bf Claim  }  There exist $c_1,\theta$ such that if
$0 \leq s \leq 1/2$ and $\bfeta \in \pairs_{n+s}$,
then
\[       \sum_{\bfeta' \in \Sep_{n+1}}
\lambda(\bfeta'\mid \bfeta) \geq c_1 \, \Delta(\bfeta)^{\theta}.\]
\end{itemize}

This is obtained by repeated application of the previous
claim.

\begin{itemize}    
\item {\bf Claim}  There exist $\rho < 1 $ such that if
$\bfeta \in \pairs_{n}$, then 
\[   
\sum_{\bfeta' \in \pairs_{n+ 4 r}, \Delta(\bfeta') <r}
\lambda(\bfeta'\mid \bfeta)   \leq \rho.\]

\end{itemize}
Indeed, using Proposition \ref{basiclerw} and Lemma \ref{importantlemma},
we can see that
\[     \sum_{\bfeta' \in \pairs_{n+ 4 r}, \Delta(\bfeta') <r}
\mu(\bfeta'\mid \bfeta)   \leq \rho .\]

We choose $N$ sufficiently large so that
\[              \sum_{j=N}^\infty j^2 \, e^{-j}
          \leq \frac 14.\]
Let $u_j = 1/2$ for $j \leq N$ and for $j > N$,
\[               u_j = u_{j-1} - j^2 \, e^{-j} .\]
in particular, $1/4 \leq u_j \leq 1/2$ for all $j$.
Let $b_j$ be the infimum of
\[ \frac{ \lambda_{n+1}^\Sep(\bfeta)}
  {\lambda_{n+1} (\bfeta)}  , \]
where the infimum is over all $n \leq u \leq n+ u_j$,
and all $\bfeta \in \pairs_{n+u}$ with 
\[              \Delta(\bfeta) \geq e^{-j}.\]
Note that the result will hold with
\[                 c = \inf_j b_j , \]
so we need only show that the right-hand side is strictly
positive.
The result is obtained by establishing two facts:
\begin{enumerate}
\item  There exist $\alpha > 0, c_0 >0$ such that
$b_j \geq c_0\, e^{-\alpha j}$. In particular, 
for each $j$, $b_j > 0$.
\item  There exists a summable sequence $\delta_j$ such
that for all $j$ sufficiently large,
\[           b_{j+1} \geq b_j \, [1 - \delta_j].\]

\end{enumerate}
Indeed, if these hold and    $M$ is sufficiently large so that
$\delta_j \leq 1/2$ for $j \geq M$, then
\[         \inf_{j} b_j  \geq  b_M \, \prod_{j = M}^\infty
(1-\delta_j).\]         
The proof proceeds now as in \cite{LV}.


\begin{thebibliography}{00}

\bibitem{BLV}  C. Bene\v{s}, G. Lawler, F. Viklund,
Scaling limit of the loop-erased random walk Green's function,
to appear in Prob. Th. Rel. Fields.

\bibitem{Kenyon}   R. Kenyon (2000).  The
asymptotic determinant of the discrete Laplacian, Acta Math. {\bf 185}, 239--286.

\bibitem{Kesten} H. Kesten (1987). Hitting probabilities of random walks on $\Z^d$,
Stoc. Proc. and
Appl. {\bf 25}, 165--184.

\bibitem{Kozma}  G. Kozma (2007). 
The scaling limit of loop-erased random walk in three dimensions, Acta Math. {\bf 19},
29-152.

\bibitem{LLERW}
G. Lawler (1980).  A self-avoiding
random walk,   Duke Math   J  {\bf 47}, 655-694.

\bibitem{cutpoint}  G. Lawler (1996).
Hausdorff dimension of cut points for Brownian motion, Electr  J.   Probab. {\bf 1},
paper no. 2.

\bibitem{rwcutpoint}   G. Lawler (1996). Cut times for simple
random walk, Electron. J. Probab. {\bf 1}, paper no. 13.

\bibitem{Nonintersect} G. Lawler (1995). 
Nonintersecting planar Brownian motions, Mathematical Physics Electronic Journal {\bf 1},
paper \#4.



\bibitem{LL}  G. Lawler, V. Limic (2010).  {\em Random Walk: A Modern Introduction},
Cambridge U. Press.

\bibitem{LPuckette} G. Lawler, E. Puckette (1997).
The disconnection exponent for simple random walk, Israel J.  Math.
{\bf 99}, 109--122.

\bibitem{LSWexp} G. Lawler, O. Schramm, and W. Werner (2001). 
Values of Brownian intersection exponents II: plane exponents,
Acta Math. {\bf 187}, 275--308.

\bibitem{LSW}  G. Lawler, O. Schramm, W. Werner (2004).
{Conformal invariance of planar loop-erased
   random walks and uniform spanning trees}, Ann.\ Probab.\ {\bf 32},   939--995.

\bibitem{LSunW}  G. Lawler, X. Sun, W. Wei,
Loop-erased random walk, uniform spanning forests and bi-Laplacian
Gaussian field in the critical dimension, preprint.

\bibitem{LV} G. Lawler, F. Viklund,
Convergence of loop-erased random walk in the natural parametrization, preprint.

\bibitem{Vermesi}
G. Lawler, B. Vermesi (2012).
Fast convergence to an invariant measure for non-intersecting 3-dimensional
Brownian paths, Alea, Latin Amer. J of Prob. and Stat. {\bf 9}, 717–738.

\bibitem{Masson}  R. Masson (2009) The growth exponent
for loop-erased random walk, Electr. J. Probab. {\bf 14}, paper no. 36,
1012-1073.

\bibitem{Shir} D. Shiraishi, Growth exponent for loop-erased walk
in three dimensions, preprint.





\end{thebibliography}
\end{document}